\documentclass[12pt]{amsart}
\usepackage{geometry}
\usepackage[english]{babel}
\usepackage[utf8]{inputenc}
\usepackage{csquotes,comment,dirtytalk} 
\usepackage{graphicx,epstopdf,float} 
\usepackage{framed} 
\usepackage{amsmath,amsthm,amssymb,old-arrows,mathtools,extpfeil}
\usepackage{dirtytalk}
\usepackage{tikz, tikz-cd}
\usepackage{xfrac} 
\usepackage{mathrsfs} 
\usepackage[pdftex]{hyperref} 
\usepackage{cancel} 
\usepackage{bm}

\usepackage{todonotes}

\newcommand{\frob}{\operatorname{Frob}}
\newcommand{\Gal}{\operatorname{Gal}}
\newcommand{\Q}{\mathbb{Q}}
\newcommand{\Qb}{\overline{\mathbb{Q}}}
\newcommand{\gl}{\mathrm{GL}}
\renewcommand{\sl}{\operatorname{SL}}
\newcommand{\Z}{\mathbb{Z}}
\newcommand{\K}{\mathbb{K}}
\newcommand{\bb}[1]{\mathbb{#1}}

\newcommand{\tr}{\operatorname{tr}}

\numberwithin{equation}{section} 

\hypersetup{
colorlinks,
citecolor=black,
filecolor=black,
linkcolor=black,
urlcolor=black
}

\newtheorem{theorem}{Theorem}[section]
\newtheorem*{theorem*}{Theorem}

\newtheorem{proposition}[theorem]{Proposition}

\newtheorem{definition}[theorem]{Definition}
\newtheorem*{definition*}{Definition}

\newtheorem{lemma}[theorem]{Lemma}
\newtheorem*{lemma*}{Lemma}

\newtheorem*{fact*}{Fact}

\newtheorem{conjecture}[theorem]{Conjecture.}
\newtheorem*{conjecture*}{Conjecture.}

\newtheorem*{corollary*}{Corollary}

\theoremstyle{definition}
\newtheorem{remark}[theorem]{Remark}

\newtheorem*{assumption*}{Assumption}
\newtheorem{assumption}{Assumption}

\usepackage[style=alphabetic]{biblatex}
\DeclareFieldFormat[article]{volume}{\mkbibbold{Vol. #1}}
\DeclareFieldFormat[article]{title}{\mkbibemph{#1}}
\DeclareFieldFormat[article]{journaltitle}{{#1}}
\DeclareFieldFormat[book]{volume}{\mkbibbold{Vol. #1}}
\DeclareFieldFormat[book]{title}{\mkbibemph{#1}}
\DeclareFieldFormat[book]{journaltitle}{#1}
\DeclareFieldFormat[incollection]{volume}{\mkbibbold{Vol. #1}}
\DeclareFieldFormat[incollection]{title}{\mkbibemph{#1}}
\DeclareFieldFormat[incollection]{journaltitle}{#1}

\title[Bogomolov property for modular Galois representations with...]{Bogomolov property for modular Galois representations with nontrivial nebentypus}
\author{Pietro Piras}
\date{March 2026}

\begin{document}

\begin{abstract}
    A field in which the (absolute, logarithmic) Weil height is bounded from below by a positive constant is said to have the Bogomolov property (property (B)). Given a normalized eigenform $f\in S_k(\Gamma_0(N))$, Amoroso and Terracini \cite{AT25} proved (B) for the field \say{cut out} by the adelic representation associated to $f$ under certain assumptions on $f$, generalizing the earlier work of Habegger \cite{Hab13} on elliptic curves. In this paper we extend this result to the case of normalized eigenforms with nontrivial nebentypus character. We also introduce the notion of ADZ field, inspired by earlier work of \citeauthor{ADZ14} \cite{ADZ14}, exhibiting a class of fields in which property (B) is preserved under arbitrary field composita.
\end{abstract}
\maketitle
\tableofcontents
\section{Introduction}\label{sec1}
Let $h:\Qb\rightarrow \bb{R}$ be the (absolute, logarithmic) Weil height (see Definition \ref{def:height}). In the seminal paper \cite{BZ01} the authors define the \emph{Bogomolov property} for any set $\mathcal{A}\subseteq \Qb$:
\begin{displayquote}
    \emph{A set $\mathcal{A}\subseteq \Qb$ has the Bogomolov property (property (B) for short) if there exists a constant $C>0$ such that for all nonzero elements $\alpha\in\mathcal{A}$ that are not roots of unity, we have $h(\alpha)>C$.}
\end{displayquote}
In the literature, property (B) has been extensively investigated when $\mathcal{A}$ is a Galois extension of the rationals (but see \cite{DK25} for a very recent result for non-Galois extensions). The question of whether a field has property (B) is related to the following conjecture of Lehmer.
\begin{conjecture}[\cite{Leh33}]\label{conj:lehmer}
    There exists a constant $C>0$ such that for all nonzero $\alpha\in\Qb$ that are not roots of unity, one has
    \begin{equation*}
        h(\alpha)\geq \frac{C}{[\Q(\alpha):\Q]}.
    \end{equation*}
\end{conjecture}
Fields in which Lehmer's conjecture holds are said to have the Lehmer property and, of course, property (B) implies the Lehmer property. For a review of Lehmer's problem see Smyth's excellent survey \cite{Smy08}. We also report the following generalization of Conjecture \ref{conj:lehmer} by Rémond.
\begin{conjecture*}[\cite{Rem17}, Conjecture 3.4]
    Let $\Gamma$ be a finitely generated subgroup of $\overline{\bb{Q}}^\times$ and let $\Gamma_{div}=\{\gamma\in\overline{\bb{Q}}^\times\mid \gamma^n\in \Gamma \text{ for some } n\in\bb{N}\}$ be its division group. There is a positive constant $C$ such that
  \begin{equation*}
    h(\alpha)\geq \frac{C}{[\bb{Q}(\Gamma_{\operatorname{div}},\alpha):\bb{Q}(\Gamma_{\operatorname{div}})]}.
   \end{equation*}
\end{conjecture*}
\noindent For more on Rémond's conjecture, see for example \cite{Rem17,Amo16, Gri17,Pot21,Ple24b,Ple22,Ple24a}. We now give a list of some fields with property (B).
\begin{enumerate}
    \item By a theorem of Northcott (see \cite[Theorem 1.6.8.]{BG06} for example), there are finitely many algebraic numbers with bounded height and bounded degree. Northcott's theorem implies that any number field has property (B) since there are only finitely many elements of $\bb{K}$ of bounded height.
    \item In \citeyear{Sch73} Schinzel \cite{Sch73} showed that the field of totally real numbers $\Q^{tr}$ has property (B). \citeauthor{Gar07} generalized this as follows
    \begin{theorem*}[\cite{Gar07}, Theorem 1]
        Let $\alpha\neq 0,\pm 1$ be an algebraic number of degree $d$ with $r>0$ real Galois conjugates. Then
        $$
            h(\alpha)\geq \frac R2 \log\left(\frac{2^{1-1/R}+\sqrt{4^{1-1/R}+4}}{2}\right)
        $$
        where $R=r/d$.
    \end{theorem*}
    In the case $R=1$ (i.e. $\alpha$ is totally real) we obtain the previously cited result of Schinzel
    \begin{equation*}
        h(\alpha)\geq \frac12 \log\left(\frac{1+\sqrt{5}}{2}\right).
    \end{equation*}
    Moreover one can prove Schinzel's result using Bilu's equidistribution theorem \cite{Bi97}, which roughly states that points of small height equidistribute (in a weak sense) on the unit circle inside the complex numbers.
    \item In \cite{BZ01}, the authors show a $p$-adic analogue of Schinzel’s theorem above: the maximal algebraic extension of $\Q$ in which the prime $p$ is totally split has property (B).
    \item The maximal abelian extension of a number field $\bb{K}$ has property (B) (see \cite{AD00} for the case $\bb{K}=\bb{Q}$ and \cite{AZ00} for the general case).
    \item Recall that a group is said to be \emph{central-by-finite} if its center has finite index. By \cite[Theorem 1.5.]{ADZ14}, Galois extensions of number fields with central-by-finite Galois group have property (B).
    \item A celebrated result of Habegger \cite{Hab13} states that the field $\Q(E_{tors})$, obtained by adjoining the coordinates of torsion points of an elliptic curve $E$ defined over $\bb{Q}$ has property (B). Recently, Sahu has showed in a preprint \cite[Corollary 1.5]{Sah25} that if $E$ is an elliptic curve defined over a number field $\bb{K}$ with at least one real embedding, then $\bb{K}(E_{tors})$ has property (B).
\end{enumerate}
\begin{definition}[Definition 1.1, \cite{AT25}]\label{def:propBmor}
    Let $\Omega$ be a separated topological group and $\mathbb{F}$ a subfield of $\overline{\Q}$. Given a continuous homomorphism
    \begin{equation*}
      \rho: \Gal(\Qb/\mathbb{F}) \longrightarrow \Omega
    \end{equation*}
    the field $\mathbb{F}(\rho):=\Qb^{\operatorname{ker}\rho}$, i.e. the field fixed by the kernel of $\rho$, is called the field \say{cut out} by the representation $\rho$. We say that $\rho$ has property (B) if the field cut out by $\rho$ has property (B).
\end{definition}
One can restate some classical results on property (B) using the language of Galois representations:
\begin{enumerate}
    \item Every homomorphism of $\Gal(\overline{\Q} / \mathbb{Q})$ taking values in an abelian group has property (B). This is indeed a restatement of \citeauthor{AD00} \cite{AD00} where they show that every abelian extension of $\Q$ has property (B). Taking $\Gal(\overline{\Q} / \mathbb{K})$ instead of $\Gal(\overline{\Q} / \mathbb{Q})$ with $\bb{K}$ a number field one obtains the main result of \cite{AZ10}.
    \item Let $E$ be an elliptic curve defined over $\Q$ and $T_{\ell}({E})$ is the corresponding $\ell$-adic Tate module. Consider the map
    \begin{equation*}
        \rho_{E}: \Gal(\overline{\Q} / \mathbb{Q}) \longrightarrow \mathrm{GL}_2(\widehat{\mathbb{Z}})
    \end{equation*}
    which is the product of all the $\ell$-adic representations
    \begin{equation*}
        \rho_{{E}, \ell}: \Gal(\overline{\Q} / \Q) \longrightarrow \operatorname{Aut}\left(T_{\ell}({E})\right)\simeq\mathrm{GL}_2\left(\mathbb{Z}_{\ell}\right).
    \end{equation*}
    The field fixed by the kernel of each $\rho_{E,\ell}$ is $\Q(E[\ell^\infty])$ and the field fixed by $\rho_E$ is $\Q(E_{tors})$ (see \cite{AT25}). Habegger's result \cite{Hab13} is that the representation $\rho_E$ has property (B).
    \item Let $\bb{K}$ be a number field and
    \begin{equation*}
        \rho:\Gal(\Qb/\bb{K})\rightarrow \gl_d(\bb{Z}_p)
    \end{equation*}
    a representation. In \cite{CT25} the authors show that if $\rho$ maps an inertia subgroup at a prime above $p$ surjectively onto an open subgroup of $ \gl_d(\Z_p)$, then $\rho$ has property (B). With this result they provide many new examples of fields with property (B).
\end{enumerate}
Let $S_k(\Gamma_0(N))$ denote the space of cuspidal forms for $\Gamma_0(N)$ of weight $k$. In \cite{AT25} the authors show that under some assumptions, given a compatible system of representations $\{\rho_{f,v}\}_{v}$ associated to a newform $f\in S_k(\Gamma_0(N))$ with $k\geq 2$, the field cut out by the product representation over all finite places
\begin{equation*}
    \widehat{\rho}_f=\prod_v \rho_{f,v}
\end{equation*}
has property (B). The goal of this paper is to improve the main result of \cite{AT25} by allowing nontrivial nebentypus. We give now state the main theorem in a qualitative version.
\begin{theorem*}[Theorem \ref{thm:main}]
    Let $f$ be a newform of weight $k$, level $N$ and nebentypus $\chi$ with $q$-expansion $$f=\sum a_nq^n.$$ Let $\bb{E}$ be the coefficient field of $f$, $\mathcal{O}_\bb{E}$ its ring of integers and $\widehat{\mathcal{O}}_\bb{E}$ its profinite completion. If $p\geq 2k-1$ is a prime not dividing the level such that $a_p=0$ and the mod $p$ representation associated to $f$ contains $\sl_2(\mathcal{O}_\bb{E}/p\mathcal{O}_\bb{E})$, then the adelic representation 
    \begin{equation*}
        \widehat{\rho}_f:\Gal(\overline{\Q} / \mathbb{Q})\rightarrow \gl_2(\widehat{\mathcal{O}}_\bb{E}).
    \end{equation*}
    has property (B).
\end{theorem*}
\begin{remark}
    One could ask what happens when $k = 1$. In that case, one does not have a compatible system of representations, but instead a complex representation whose image is known to be finite (see \cite{DS74}). Hence, the field cut out by this representation is a number field and has property (B) by Northcott’s theorem.  
\end{remark}
The main difficulty in proving Theorem \ref{thm:main} is the existence of inner twists of $f$ that complicate the image of the representation. To resolve this issue one uses that, up to passing to a finite extension of the coefficient field of $f$, for sufficiently large $p$ the image of the representation at $p$ is controlled. In fact, one shows that a finite extension of the field cut out by the representation has property (B).\\
To study the representation outside of $p$ we introduce the notion of ADZ field, inspired by \cite{ADZ14}. Given a number field $\bb{K}$, a Galois extension $\bb{F}/\bb{K}$ is ADZ at a place $v$ of $\bb{K}$ if the subfield of $\bb{F}$ fixed by the center of $\Gal(\bb{F}/\bb{K})$ has (uniformly) bounded local degrees over $v$. In \cite[Theorem 1.5]{ADZ14} the authors show that ADZ fields have property (B) \emph{uniformly} in $v$, the degree $[\bb{K}:\bb{Q}]$ and the bound on the local degrees over $v$, but this property has not been further investigated. It turns out that being ADZ is preserved under field compositum i.e. the compositum of ADZ fields at $v$ at is again ADZ; this is in stark contrast with the general bad behavior of property (B) under the compositum of fields, since even finite extensions of fields with property (B) need not have property (B) (see for example \cite[Theorem 5.3]{ADZ14}). The proof is entirely group-theoretic and rests upon the fact that one can describe the center of the Galois group of the compositum in terms of the centers of the other Galois groups (see Lemma \ref{lemma:centro}). Moreover, since there is only a finite number of extensions of bounded degree of a local field, property ADZ is preserved under \emph{arbitrary} composita of ADZ fields at $v$ with uniformly bounded local degrees over $v$ (Theorem \ref{thm:compostoadz}). One shows (Proposition \ref{prop:adz}) that the representation outside of $p$ has property (B) when restricted to an ADZ field at a place above $p$ and in particular to a number field. ADZ fields therefore seem to be a well behaved class of fields with respect to property (B); for example, since every abelian extension $\bb{L}/\bb{K}$ is ADZ, the compositum of $\bb{L}$ with an abelian extension is ADZ and therefore has property (B) (see Remark \ref{rem:adz}).
\\
\\
{\bf Outline of the paper.} In Section \ref{sec2} we recall the Deligne-Shimura construction of the system of compatible Galois representations attached to a normalized eigenform $f\in S_k(\Gamma_0(N),\chi)$ with $k\geq 2$ and its main properties, such as its values on Frobenius elements.\\
In Section \ref{sec3} we state the main result of the paper (Theorem \ref{thm:main}) and discuss the hypotheses. Note that one cannot trivially apply the methods of \cite{AT25} in our context i.e. in the case of nontrivial nebentypus $\chi$, since the nontriviality of $\chi$ produces inner twists of $f$ and the hypotheses of \cite{AT25} may not be met. One has to take into account that, even after descending to a representation in the ring of traces (Carayol's result), this ring will be larger than $\bb{Z}_p$, as was instead in \cite{AT25}, because one has to adjoin the image of a crystalline character $\varpi$. We therefore modify the assumptions of \cite{AT25} on the image of the mod $p$ representation accordingly.\\
Section \ref{sec4} contains the proof of the main result, and is divided into three parts. We begin by recalling some results of \cite{ADZ14} about lower bounds for the height and the \say{acceleration lemma} \cite[Lemma 2.1]{ADZ14}; we also state the supplementary assumptions needed to prove the \say{merging theorem} Theorem \ref{thm:composto}. This is a slight generalization of \cite[Proposition 3.4.]{AT25} where the authors show that if $\bb{F}/\bb{Q}$ is unramified at $p$ and has property (B), and $\bb{L}$ satisfies some conditions (Assumptions 3.2 and 3.3 \emph{loc.cit.}), then $\bb{FL}$ has property (B). Theorem \ref{thm:composto} is an analogous statement where $\bb{F}$ and $\bb{L}$ are extensions of a number field $\bb{K}$. In the following subsection we analyze the prime to $p$ part $\rho^{(p)}$ of $\widehat{\rho}_f$ and give the definition of an \emph{ADZ field} (Definition \ref{def:adz}). We show that after restricting to an ADZ field at a place $v$ above $p$, the representation $\rho^{(p)}$ has property (B). After dealing with $\rho^{(p)}$ we turn to the $p$-part of $\widehat{\rho}$ i.e. the representation $\rho_p$. By results of Breuil and Scholl we see that under our assumptions on $f$, the restriction of $\rho_p$ to $G_{\bb{Q}_p}$ is a crystalline representation $V_{k,0,\varpi}$ which allows us to study $\Q_p(\rho_p)$ in a manner similar to \cite{AT25}. Finally, by assuming the \say{normal closure lemma} (Proposition \ref{prop:ncl}), we prove Theorem \ref{thm:main} by applying Theorem \ref{thm:composto} to the fields cut out by $\rho_p$ and $\rho^{(p)}$. \\
In Section \ref{sec5} we prove the normal closure lemma completing the proof of the main theorem.\\
The last section is devoted to examples. We use an unpublished result of \citeauthor{Mas22} to produce modular forms and representations that satisfy our assumptions. For this we also need to guarantee that our representation mod $p$ does not come from a representation associated to a modular form of lower level. To do this we recall Serre's conjectures on representations coming from modular forms and show that, under the hypotheses $a_p=0$, the minimal weight associated to the mod $p$ reduction of $\rho_v$ with $v\mid p$ is eihter $k$ or 2.
\\
\\
{\bf Notations.}\label{subs:notations}
Given a group $G$ its \emph{center} will be denoted by $Z(G)$. Algebraic number fields will be denoted with bold letter such as $\bb{K,F,L}\ldots$ while local fields by $K,L,F\ldots$ For every field (global or not) $\bb{K}$ inside an algebraic closure $\overline{\bb{K}}$ we will denote by $G_\bb{K}$ the Galois group of $\overline{\bb{K}}/\bb{K}$; the ring of integers of $\bb{K}$ is denoted by $\mathcal{O}_\bb{K}$. If $\mathfrak{p}$ is a nonzero prime ideal of $\mathcal{O}_\bb{K}$, we will say that $\mathfrak{p}$ is a \say{prime of $\bb{K}$} instead of \say{a prime ideal of the ring of integers of $\bb{K}$}.\\
For any ring $R$ we denote by $M_n(R)$ the ring of $n\times n$ matrices with coefficients in $R$, $\gl_n(R)\subseteq M_n(R)$ the matrices that have invertible determinant and $\sl_n(R)\subseteq \gl_n(R)$ the matrices with determinant 1.\\
Throughout the article the $p$-adic cyclotomic character is the character
\begin{align*}
    \varepsilon_p:G_\Q&\rightarrow \Z_p^\times\\
    \sigma&\mapsto (\chi_{p^n}(\sigma))_n
\end{align*}
where for each integer $n\geq 0$, $\chi_{p^n}(\sigma)$ satisfies  $\sigma(\zeta_{p^n})=\zeta_{p^n}^{\chi_{p^n}(\sigma)}$ for all $\sigma\in G_\Q$ and for all primitive $p^n$-roots of unity $\zeta_{p^n}$.\\
If $\bb{K}$ is a number field and $v$ a place of $\bb{K}$, the completion of $\bb{K}$ with respect to $v$ will be denoted by $\bb{K}_v$. If $v$ is a finite place above the rational prime $p$, the field $\bb{K}_v$ is a local field with ramification index $e$ and residual degree $f$. Define the absolute value $|\cdot|_v$ on $\bb{K}_v$ as
\begin{equation*}
    |x|_v=p^{-v(x)/e}
\end{equation*}
for all $x\in \bb{K}_v^\times$, and $|0|_v=0$. If $v$ is instead an infinite place of $\bb{K}$ associated to the set of embeddings $\{\sigma,\overline{\sigma}\}$, then $\bb{K}_v$ is either $\bb{R}$ or $\bb{C}$ and the absolute value on $\bb{K}_v$ is $|\cdot|_v=|\sigma(\cdot)|$. Given a nonzero element $x$ of $\K$, the product formula reads 
\begin{equation*}
    \prod_v|x|^{[\bb{K}_v:\bb{Q}_v]}_v=1
\end{equation*}
where $v$ runs over all the places of $\bb{K}$.\\
If $v$ and $w$ are valuations, we will write $w\vert v$ to denote that $w$ extends $v$.
Now let $K$ be a local field with valuation $v$ and $L/K$ be a finite Galois extension with valuation $w$ extending $v$. Let $i\geq -1$ be an integer. We define the $i$-th ramification group of $L/K$ as
\begin{equation*}
    G_i=\{\sigma \in\Gal(L/K)\mid w(\sigma(x)-x)\geq i+1 \text{ for all }x\in\mathcal{O}_L \}.
\end{equation*}
We have $G_{-1}=\Gal(L/K)$ and $G_0$ is called the inertia group, which we denote by $I$.\\
Let $\mathfrak{P}$ be a maximal ideal of $\Qb$ above the rational prime $p$. A Frobenius element $\frob_\mathfrak{P}$ at $\mathfrak{P}$ is any element of $\Gal(\Qb/\Q)$ such that its image under the reduction map
\begin{equation*}
    G_\bb{Q}\rightarrow G_{\bb{F}_p}
\end{equation*}
is the usual Frobenius map i.e. the map that raises an element to the $p$-th power. If $\bb{K}/\Q$ is a Galois extension and $\mathfrak{p}$ is a prime of $\bb{K}$ above the rational prime $p$, then we denote the restriction of $\frob_{\mathfrak{P}}$ to $\bb{K}$ by $\frob_\mathfrak{p}$.\\
To give an embedding $\iota: \Qb\hookrightarrow\Qb_p$ is equivalent to choosing a maximal ideal $\mathfrak{P}$ of $\Qb$, or equivalently a valuation on $\Qb$. Choosing $\mathfrak{P}$ also determines, for all subfields $\bb{K}$ of $\Qb$, the prime $\mathfrak{p}=\mathfrak{P}\cap\mathcal{O}_{\bb{K}}$ and the valuation associated to $\mathfrak{p}$ is the restriction of $v$ to $\bb{K}$. We speak of \say{the ideal/valuation of $\bb{K}$ determined by $\iota$} to denote $\mathfrak{p}$ or the restriction of $v$ to $\bb{K}$. \\
The greatest common divisor of $N,M\in\bb{Z}$ is denoted by $(N,M)$.\\
If $S$ is a set, we denote its cardinality by $|S|$.\\
The set of cuspidal modular forms for the group $\Gamma_0(N)$ with nebentypus $\chi$ is denoted by $S_k(\Gamma_0(N),\chi)$.\\
\\
{\bf The Weil height.} We now recall the definition of the (absolute, logarithmic) Weil height.
\begin{definition}\label{def:height}
The Weil height is the function $h:\Qb\rightarrow \bb{R}_{\geq 0}$ defined as
\begin{equation*}
    h(\alpha) = \frac{1}{[\bb{K}:\mathbb Q]} \sum_{v} [\bb{K}_v:\mathbb Q_v]\log\max\bigl\{|\alpha|_v,1\bigr\}
\end{equation*}
for all nonzero algebraic numbers $\alpha$, where $\K$ is any number field that contains $\alpha$ as $v$ ranges over the places of $\K$. {We put $h(0)=0$.}
\end{definition}
We list some properties of the Weil height (see \cite{Zan18}).
\begin{proposition}\label{eqn:propaltezza}
    For $x_1,\ldots,x_n\in\Qb$ and $\alpha\in \Qb$ the Weil height satisfies the following properties
    \begin{enumerate}
    \item $h(x_1\cdots x_n)\leq h(x_1)+\cdots + h(x_n)$.
    \item $h(x^m)=|m|h(x)$ for any $x\in\Qb$ and any $m\in\Z$ (so $h(1/x)=h(x)$ for all nonzero $x\in\Qb$).
    \item $h(\sigma(x))=h(x)$ for each $\sigma\in G_\bb{Q}$.
    \item We have $h(\alpha)=0$ if and only if $\alpha$ is a root of unity.
    \end{enumerate}
\end{proposition}
\section{The Deligne-Shimura construction}\label{sec2}
In this section we recall the main properties of the compatible system of representations attached to modular forms. The case $k=2$ is due to Shimura (see for example \cite[Section 9.5]{DS05} for the explicit construction) while the case $k>2$ is due to Deligne \cite{Del71} and uses étale cohomology.
Let $f \in S_k(\Gamma_0(N),\chi)$ be a newform for the Hecke algebra of weight $k\geq 2$ and level $N$ with $q$-expansion $f=\sum_{n=1}^\infty a_n(f)q^n$. As is known classically (see \cite[pg 21]{Rib77} for example) the field $\mathbb{E}=\Q(\{a_n(f)\}_{n\geq 1})$ is a number field called the \emph{coefficient field} or the \emph{Hecke field} of $f$. The following theorem associates to $f$ a compatible system of Galois representations.
\begin{theorem}[\cite{DS05}, Theorem 9.6.5]
    Let $f$ and $\bb{E}$ be as above, and let $\ell$ be a prime. For each place $v$ of $\bb{E}$ above $\ell$ there is an irreducible 2-dimensional Galois representation
    \begin{equation*}
        \rho_v:G_\Q\rightarrow \gl_2(\bb{E}_v).
    \end{equation*}
    This representation is unramified at all primes $p\nmid \ell N$ i.e. the image of inertia groups at primes above $p \nmid \ell N$ is trivial. For any prime ideal $\mathfrak{P}$ of $\Qb$ above such $p$, the characteristic polynomial of $\rho_v(\frob_{\mathfrak{P}})$ is
    \begin{equation*}
        X^2-a_p(f)X+\chi(p)p^{k-1}.
    \end{equation*}
\end{theorem}
We call the set $\{\rho_v\}_v$ as $v$ ranges over all the places of $\bb{E}$ above $\ell$ a \emph{compatible system of $\ell$-adic Galois representations}. The characteristic equation shows that for all the $\ell$-adic representations $\rho_v$ in the compatible system, we have
\begin{equation*}
    \det\rho_v=\varepsilon_\ell^{k-1}(\chi\circ \varepsilon_\ell).
\end{equation*}
One can collect all the local representations of the compatible systems as $\ell$ varies. The global \say{adelic} representation $\widehat{\rho}$ is defined by taking the product over all the finite places $v$ of $\mathcal{O}_\bb{E}$
\begin{equation*}
    \widehat{\rho}:= \prod_v \rho_{v} :G_\Q\rightarrow \prod_v\gl_2(\mathcal{O}_{\bb{E},v}) \simeq \gl_2(\widehat{\mathcal{O}}_\bb{E}).
\end{equation*}
where $\widehat{\mathcal{O}}_\bb{E}$ denotes the profinite completion of ${\mathcal{O}}_\bb{E}$.
\section{Main result}\label{sec3}
Fix a newform for the Hecke algebra $f\in S_k(\Gamma_0(N),\chi)$ of weight $k\geq 2$ with $q$-expansion $f=\sum_{n=1}^\infty a_nq^n$ and nontrivial nebentypus $\chi$. We also assume that $N>3$, otherwise the conductor of the nebentypus $\chi$ is either 1 or 2 and $\chi$ is the trivial character. As before we put $\mathbb{E}=\Q(\{a_n(f)\}_{n\geq 1})$ and $\mathcal{O}_\bb{E}$ its ring of integers. For each rational prime $p$ denote by $\rho_{p}$ the representation \say{at $p$} i.e.
\begin{equation}\label{eqn:rhoinsidep}
    \rho_{p}=\prod_{v \mid p} \rho_{v}: G_{\mathbb{Q}} \longrightarrow \prod_{v\vert p} \gl_2(\mathcal{O}_{\bb{E}_v})
\end{equation}
and by $\rho^{(p)}$ the part \say{outside of $p$} of $\widehat{\rho}$ i.e.
\begin{equation}\label{eqn:rhoutsidep}
    \rho^{(p)}=\prod_{v \nmid p} \rho_{v}: G_{\mathbb{Q}} \longrightarrow \prod_{v\nmid p} \gl_2(\mathcal{O}_{\bb{E},v}).
\end{equation}
If $f$ is CM (\cite[pg. 34]{Rib77}) then $\Q(\widehat{\rho})$ is an abelian extension of a quadratic number field, so it has property (B) by \cite{AZ10}. For this reason \emph{$f$ will always be assumed without complex multiplication}. The main result of \cite{AT25} was conditional on the existence of a prime $p$ such that the image of $\rho_{p}$ is 
$$
    \widehat{G}(\mathcal{O}_\bb{E}/p\mathcal{O}_\bb{E})=\{u\in\gl_2(\mathcal{O}_\bb{E}/p\mathcal{O}_\bb{E})\mid \det u=\alpha^{k-1}, \alpha\in\bb{F}_p^{\times}\}.
$$
For such prime $p$, the image of $\rho_{p}$ contains $\sl_2(\mathcal{O}_\bb{E}/p\mathcal{O}_\bb{E})$ which is used to prove the \say{normal closure lemma} \cite[Proposition 5.1.]{AT25}. In the case of non trivial nebentypus some problems arise due to the existence of \emph{inner twists} of $f$ that complicate the image of $\rho_{p}$. However, as we will see, as long as one considers a normal subgroup of finite index of $G_\Q$ where $\chi$ is trivial, the image of $\widehat{\rho}$ is known and has been studied by Momose, Papier and Ribet. We now recall some general results about inner twists following \cite{Rib85}, \cite{Rib80} and \cite{Rib77}.\\
\\
{\bf Inner twists.}
Let $f$ and $\bb{E}$ be as before and let $\operatorname{Aut}(\bb{E}/\bb{Q})$ be the set of automorphisms of $\bb{E}$ fixing $\bb{Q}$. Let $\gamma\in \operatorname{Aut}(\bb{E}/\bb{Q})$ and consider
\begin{equation*}
    \gamma f=\sum_{n\geq 1} \gamma(a_n)q^n.
\end{equation*}
This is again a newform (see \cite{Rib77} end of page 21) of the same level of $f$ with nebentypus $\gamma\circ\chi$. If there exists a Dirichlet character $\delta$ such that
\begin{equation}\label{eqn:inner}
    \gamma(a_p)=\delta(p)a_p
\end{equation}
for all but finitely many primes $p$, we call the couple $(\gamma,\delta)$ an inner twist of $f$. Since we have assumed that $f$ is not CM, for every $\gamma\in \operatorname{Aut}(\bb{E}/\bb{Q})$ there is at most one Dirichlet character $\delta$ such that $(\gamma,\delta)$ is an inner twist of $f$ (\cite[Section 3]{Rib80}).
\begin{remark}\label{rmk:neben}
    Since $f$ is not CM, the coefficient field $\bb{E}$ is either totally real or a CM field (\cite[Proposition 3.2]{Rib77}) and the complex conjugation $\overline{\cdot}$ is an automorphism of $\bb{E}$. We have that
    \begin{equation*}
        a_p=\overline{a_p}\chi(p)
    \end{equation*}
    for all primes $p$ not dividing $N$ (see \cite[Section 1]{Rib77}), so $(\overline{\cdot},\overline{\chi})$ is always an inner twist of $f$.
\end{remark}
Let now $\Gamma\leq\operatorname{Aut(\bb{E}/\bb{Q}})$ be the abelian subgroup (see \cite[Proposition 3.3]{Rib80}) of automorphisms $\gamma\in \operatorname{Aut}(\bb{E}/\bb{Q})$ for which there exists a character $\delta$ such that $(\gamma,\delta)$ is an inner twist of $f$. Similarly let $\Delta$ be the set of Dirichlet characters $\delta$ for which there exists $\gamma\in \Gamma$ such that $(\gamma,\delta)$ is an inner twist of $f$. Let $\bb{F}$ be the subfield of $\bb{E}$ fixed by $\Gamma$. Viewing each $\delta\in \Delta$ as a character on $G_\bb{Q}$ with values on $\bb{E}^\times$ (see \cite[Section 3]{Rib85}), let $\bb{K}$ be the subfield of $\Qb$ fixed by the intersection of all $\ker\delta$ as $\delta$ ranges in $\Delta$. Except for finitely many primes the image of $G_\bb{K}$ under $\rho_{p}$ lies in $\gl_2(\mathcal{O}_{\bb{F}}\otimes_\Z \Z_p)$ \cite[Theorem 3.1]{Rib85} and is exactly the group 
\begin{equation}\label{eqn:gathe}
    \widehat{G}(\mathcal{O}_{\bb{F}}\otimes_\Z \Z_p):=\{\gamma\in\gl_2(\mathcal{O}_{\bb{F}}\otimes_\Z \Z_p)\mid \det(\gamma)=x^{k-1},x\in\Z_p^\times\}.
\end{equation}
Let now $\eta$ be a Dirichlet character. There exists (\cite[Section 3]{Rib80}) a unique newform $g=\sum_{n\geq 1}b_nq^n$ with nebentypus $\chi\cdot \delta^2$, characterized by the relation
\begin{equation*}
  b_p=a_p(f)\eta(p)
\end{equation*}
for all but finitely many primes $p$. Moreover we have that $b_n=a_n(f)\eta(n)$ for all $n$ prime to the conductor of $\eta$. Let us write $f\otimes\eta$ for the modular form whose coefficients are twisted by the character i.e.
\begin{equation*}
    f\otimes\eta=\sum_{n\geq 1} \eta(n)a_n(f)q^n.
\end{equation*}
\begin{remark}\label{rem:Kunr}
    Fix a prime $\ell$ and consider an inner twist $(\gamma,\delta)$ of $f$ as above. Since $\gamma f$ is again a newform, we have another compatible system of $\ell$-adic representations that we denote $\{\rho^\gamma_v\}_v$ as $v$ ranges over the places of $\bb{E}$ dividing $\ell$. For each $v$ we have that
    \begin{enumerate}
        \item both $\rho^\gamma_v$ and $\rho_v\otimes \delta$ are semisimple;
        \item except for finitely many primes $p$, the traces of the image of Frobenius elements at primes above $p$ are the same.
    \end{enumerate}
    By the Chebotarev density theorem and the continuity of the trace, we conclude that the traces of $\rho^\gamma_v$ and $\rho_v\otimes\delta$ are the same and by the Brauer-Nesbitt theorem in characteristic zero \cite[Theorem 30.16]{CR88} these representation are isomorphic. Since $\rho^\gamma_v$ and $\rho_v$ are unramified at primes not dividing $N\ell$, the same is true for $\delta$. This implies that $\bb{K}$ is unramified at all primes not dividing $N$.
\end{remark}
\begin{remark}\label{rem:algebra}
    Let $R$ be a Dedekind domain that is finitely generated over $\bb{Z}$, and let $\mathcal{A}=R\otimes_\Z \Z_p$. Then $\mathcal{A}$ is isomorphic to the $(p)$-adic completion of $R$ (see \cite[Proposition 10.13]{AM69}) i.e.
    \begin{equation*}
        \mathcal{A}\simeq\varprojlim R/p^nR.
    \end{equation*}
    By the Chinese remainder theorem, letting $pR=\prod_{i=1}^k\mathfrak{p}_i^{e_in}$ we have
    \begin{equation*}
        R/p^nR\simeq\prod_{i=1}^kR/\mathfrak{p}_i^{e_in}R.
    \end{equation*}
    We have that
    \begin{equation*}
        \mathcal{A}\simeq\varprojlim R/p^nR\simeq\varprojlim_{n>0}\left(\prod_{i=1}^kR/\mathfrak{p}_i^{e_in}\right)\simeq\prod_{i=1}^k \widehat{R_i}
    \end{equation*}
    where $\widehat{R_i}$ is the completion of $R$ with respect to the the ideal $\mathfrak{p}_i$. Moreover, since for all $i$ $$\widehat{R_i}/\mathfrak{p}_i\widehat{R_i}\simeq R/\mathfrak{p}_i,$$
    we have
    \begin{equation*}
        \mathcal{A}/p^n\mathcal{A}\simeq\prod_{i=1}^k\widehat{R_i}/\mathfrak{p}_i^{e_in}\widehat{R_i}\simeq\prod_{i=1}^kR/\mathfrak{p}_i^{e_in}\simeq R/p^nR.
    \end{equation*}
\end{remark}
By Remark \ref{rem:algebra}, letting $\mathcal{A}=\mathcal{O}_\bb{F}\otimes\bb{Z}_p$ we have
\begin{equation*}
        \mathcal{A}/p^n\mathcal{A}\simeq \mathcal{O}_\bb{F}/p^n\mathcal{O}_\bb{F}
\end{equation*}
and we shall make this identification throughout the paper. We can now state the main theorem. Given a prime $p$ consider the following assumptions on $p$.
\begin{assumption}\label{ass1}
Let $\bb{E}$ be the coefficient field of $f$ and $\bb{F}\subseteq \bb{E}$ as before. Let also $\bb{K}$ be the subfield of $\Qb$ fixed by the intersection of all $\ker\delta$ as $\delta$ ranges in $\Delta$.
    \begin{itemize}
        \item[(P0)] $p \nmid N$;
        \item[(P1)] $a_p=0$;
        \item[(P2)] the image of $G_\bb{K}$ under $\rho_{p}\bmod p$ contains $\sl_2(\mathcal{O}_\bb{F}/p\mathcal{O}_\bb{F})$
        \item[(P3)] $p \geq 5,\  p \nmid k-1$ and $\frac{p+1}{2} \nmid k-1$.
    \end{itemize}
\end{assumption}
Our main theorem is
\begin{theorem}\label{thm:main}
Suppose that there exists a prime $p$ satisfying Assumption \ref{ass1}, then the adelic representation associated to $f$
    \begin{equation*}
        \widehat{\rho}_f:G_\Q\rightarrow \gl_2(\widehat{\mathcal{O}}_\bb{E})
    \end{equation*}
    has property (B).
\end{theorem}
This theorem improves the main result of \cite{AT25} where one assumes that $f$ has trivial nebentypus. We now briefly discuss \hyperref[ass1]{Assumption 1}. (P1) is motivated by the fact that if $a_p(f)=0$ and $v$ a place above $p$, the representation $\rho_{v}$ restricted to $G_{\Q_p}$ is crystalline according to \cite{Sch90}, so one can use the classification of Breuil of crystalline representations \cite[Proposition 3.1.1]{Bre03} and conclude that $\rho_v$ is isomorphic to a representation $V_{k,0,\varpi}$ described in \cite{Bre03}. Finally (P1) is used in Proposition \ref{prop:adz} to show, by looking at its characteristic polynomial, that the image of the square of a Frobenius element at $p$ under $\rho_f^{(p)}$ is scalar. In the context of elliptic curves over $\Q$ as in \cite{Hab13} this assumption was the existence of a supersingular prime. A theorem of Elkies \cite{Elk87} guarantees that for a rational elliptic curve there are actually infinitely many supersingular primes, and this has been improved to the case of number fields with at least one real embedding (see \cite{El89}). The assumption on supersingular primes was used by Habegger in combination with an open image theorem of Serre \cite{Ser72}, to guarantee the existence of infinitely many supersingular primes for $E$ such that the representation
\begin{equation*}
    G_\Q\rightarrow \operatorname{Aut}(E[p])
\end{equation*}
is surjective. (P2) is needed to guarantee that the image of the mod $p$ representation is the group $
    \widehat{G}(\mathcal{O}_{\bb{F}}\otimes_\Z \Z_p)$ appearing in \eqref{eqn:gathe}, using a group-theoretic result of Ribet (see Proposition \ref{prop:rib}). Let us mention that it is the analogue of Serre's theorem in our context of representations associated to modular forms and as in \cite{AT25} it shall be heavily used in the proof of Lemma \ref{prop:ncl} (the normal closure lemma). (P3) is also used in the proof of the normal closure lemma and to guarantee that the representation $V_{k,0,\varpi}$ is absolutely irreducible.
\section{Proof of the main result}\label{sec4}
Recall that we have fixed a newform $f\in S_k(\Gamma_0(N),\chi)$ in Section \ref{sec3}; now we fix a prime $p$ satisfying assumption \ref{ass1} and an embedding $\iota: \overline{\Q}\hookrightarrow\overline{\Q}_p$. We denote by $\mathfrak{P}$ the maximal ideal associated to this embedding and $v$ the associated valuation on $\Qb$ (see the Notations subsection). The proof of Theorem \ref{thm:main} is structured in the following way. First one proves that the representation $\rho^{(p)}$, which is the prime to $p$ part of $\widehat{\rho}$, has property (B). Since by definition $\Gal(\Q(\rho^{(p)})/\Q)$ is isomorphic to the image of $\rho^{(p)}$, using the characteristic polynomial of  $\rho^{(p)}(\frob_\mathfrak{P})$ one shows that the restriction of $\frob^2_{\mathfrak{P}}$ to $\Q(\rho^{(p)})$ is a central element of $\Gal(\Q(\rho^{(p)})/\Q)$. This is sufficient to prove that $\rho^{(p)}$ has property (B). Then one deals with the part \say{inside of $p$} of $\widehat{\rho}$, i.e. the $p$-adic representation $\rho_{p}$. This is more challenging and one has to use the fact that $\rho_{p}\vert_{\Gal(\overline{\bb{\Q}}_p/\bb{Q}_p)}$ is crystalline and a classification due to Breuil. Finally one merges the two representations using Theorem \ref{thm:composto} which is a slight generalization of \cite[Proposition 3.4]{AT25} allowing the base field to be a finite extension of $\Q$. Since $\widehat{\rho}$ is the product of $\rho_p$ and $\rho^{(p)}$, this proves that $\widehat{\rho}$ has property (B).

We now outline the strategy used to prove height bounds in infinite Galois extensions. This should explain the hypotheses of Theorem \ref{thm:probinf} which gives sufficient conditions for an infinite Galois extension of a number field to have property (B). Once Theorem \ref{thm:probinf} is established, one can easily prove the \say{merging theorem} Theorem \ref{thm:composto}. Let $\bb{K}$ be a number field, $\mathfrak{p}$ the ideal of $\bb{K}$ associated to the embedding $\iota$ and $e$ the absolute ramification index of $\mathfrak{p}$. Let $\bb{L}/\bb{K}$ be a infinite Galois extension. Suppose that $\bb{L}$ is a union of \emph{finite} Galois extensions $\bb{L}_n$ such that the height on each $\bb{L}_n$ is bounded from below by a constant independent of $n$. Since any element of $\bb{L}$ belongs to some $\bb{L}_n$ for a big enough $n$, the bound on each $\bb{L}_n$ works for every element of $\bb{L}$, so $\bb{L}$ itself has property (B). The problem is then to find sufficient conditions on each $\bb{L}_n$ to have lower bound independent of $n$. The starting point is the following slight generalization of \cite[Lemma 2.2]{ADZ14}.
\begin{lemma}\label{lemma:adzs}
    Let $\bb{K}$ be a number field and $\mathfrak{p}$ a prime of $\bb{K}$ above $p$. Let $\bb{K}^\prime/\bb{K}$ be a finite Galois extension and $\sigma\in \Gal(\bb{K}^\prime/\bb{K})$. Let $a$, $b\geq1$  be rational integers, $\xi\in\bb{R}_{>0}$ and $S$ a set of places of $\bb{K}^\prime$ above $\mathfrak{p}$. If for all $\gamma\in\mathcal{O}_{\bb{K}^\prime}$ and $v\in S$ we have
    \begin{equation*}
        \vert\gamma^a-\sigma(\gamma)^b\vert_v\leq p^{-\xi},
    \end{equation*}
    then for every $\alpha\in  \bb{K}^\prime$ such that $\alpha^a\neq\sigma(\alpha)^b$ we have
    \begin{equation}\label{eq:stima}
        h(\alpha)\geq\frac{1}{a+b}\left(|S|\frac{[\bb{K}^\prime_{v}:\bb{Q}_p]}{[\bb{K}^\prime:\bb{Q}]}\xi\log p-\log 2\right)
    \end{equation}
    where $v$ is any place in $S$.
\end{lemma}
\begin{proof}
     Let $\alpha\in \bb{K}^\prime$ such that $\alpha^a\neq\sigma(\alpha)^b$ and $v\in S$. By \cite[Lemma 1]{AD00} there exists an algebraic integer $\beta\in  \mathcal{O}_{\bb{K}^\prime}$ such that $\alpha\beta$ is also an algebraic integer and
    \begin{equation*}
        \vert\beta\vert_v=\max\{1,\vert\alpha\vert_v\}^{-1}.
    \end{equation*}
    By our assumptions we have $\vert(\alpha\beta)^a-\sigma(\alpha\beta)^b\vert_v\leq p^{-\xi}$ and $\vert\beta^a-\sigma(\beta)^b\vert_v\leq p^{-\xi}$. Using the ultrametric inequality, we deduce that
    \begin{align*}
        \vert\alpha^a-\sigma(\alpha)^b\vert_v
        &=\vert\beta\vert_v^{-a}
        \vert (\alpha\beta)^a-\sigma(\alpha\beta)^b+
        (\sigma(\beta)^b-\beta^a)\sigma(\alpha)^b\vert_v\\
        &\leq 
        c(v)\max(1,\vert\alpha\vert_v)^a\max(1,\vert\sigma(\alpha)\vert_v)^b
    \end{align*}
    with $c(v)=p^{-\xi}$. Moreover the last inequality holds for an arbitrary place $w$ of $\bb{K}^\prime$ with 
    \begin{equation*}
        c(w)=
        \begin{cases}
        1&\text{ if } w\nmid\infty, w\notin S\\
        2&\text{ if } w\mid \infty.
        \end{cases}
    \end{equation*}
    Since $\alpha^a-\sigma(\alpha)^b$ is not zero, by applying the product formula and dividing by $[\bb{K}^\prime:\Q]$ we get:
    \begin{align*}
    0 & = \sum_v\frac{[\bb{K}^\prime_v:\Q_v]}{[\bb{K}^\prime:\Q]}\log\vert \alpha^a-\sigma(\alpha)^b\vert_v\\
        & \leq \sum_v\frac{[\bb{K}^\prime_v:\Q_v]}{[\bb{K}^\prime:\Q]}\left(\log c(v)+a\log\max\{1,\vert\alpha\vert_v\}+b\log\max\{1,\vert\sigma(\alpha)\vert_v\}\right)\\
        & = \left(\sum_{v\mid\infty}\frac{[\bb{K}^\prime_v:\Q_v]}{[\bb{K}^\prime:\Q]}\right)\log 2-\left(\sum_{v\in S}\frac{[\bb{K}^\prime_v:\Q_p]}{[\bb{K}^\prime:\Q]}\right)\xi\log p + a h(\alpha)+bh(\sigma(\alpha))\\
        & =\log 2+(a+b)h(\alpha)-\xi\log p\sum_{v\in S}\frac{[\bb{K}^\prime_v:\bb{K}_v][\bb{K}_v:\bb{Q}_p]}{[\bb{K}^\prime:\Q]}.
    \end{align*}
    Since $\bb{K}^\prime/\bb{K}$ is Galois $[\bb{K}^\prime_v:\bb{K}_v]$ is independent of the place $v$ in $S$. Given a place $w$ of $S$ we obtain 
    \begin{equation*}
        \sum_{v\in S}\frac{[\bb{K}^\prime_v:\bb{K}_v][\bb{K}_v:\bb{Q}_p]}{[\bb{K}^\prime:\Q]}=\frac{|S|[\bb{K}^\prime_w:\bb{K}_w][\bb{K}_w:\bb{Q}_p]}{[\bb{K}^\prime:\Q]}=\frac{|S|[\bb{K}^\prime_w:\bb{Q}_p]}{[\bb{K}^\prime:\bb{Q}]}.
    \end{equation*}
    We deduce then the bound
\begin{align*}
    h(\alpha)&\geq \frac{1}{a+b}\left(\xi\frac{|S|[\bb{K}^\prime_w:\bb{K}_w][\bb{K}_w:\bb{Q}_p]}{[\bb{K}^\prime:\Q]}\log p-\log 2\right)\\
    &=\frac{1}{a+b}\left(|S|\frac{[\bb{K}^\prime_{w}:\bb{Q}_p]}{[\bb{K}^\prime:\bb{Q}]}\xi\log p-\log 2\right).
\end{align*}
\end{proof}
As a corollary we recover \cite[Lemma 2.2]{ADZ14} by choosing $S$ to be the set of \emph{all places} of $\bb{K}^\prime$ above $\mathfrak{p}$. Note that the previous lemma is a bound only for the special elements $\alpha$ of $\bb{K}$ such that $\alpha^a\neq \sigma(\alpha)^b$, since we need to apply the product formula. To obtain height bounds for the other elements of $\bb{K}^\prime$, we will use assumption \ref{ass2}.

Let us now come back to the previous setting of the infinite extension $\bb{L}/\bb{K}$. We want to write $\bb{L}$ as a union of finite Galois extensions $\bb{L}_n$ of $\bb{K}$ with uniform bounded height so that the union $\bb{L}$ will have property (B). Fix now $n$ and suppose that there exists a $\sigma_n\in \Gal(\bb{L}_n/\bb{K})$ such that for all $\gamma\in \mathcal{O}_{\bb{L}_n}$ and all places $v_n$ of $\bb{L}_n$ over $\mathfrak{p}$  we have 
\begin{equation}\label{eqn:stima}
    |\gamma^a-\sigma_n(\gamma)^b|_{v_n}\leq p^{-{\xi_n}}
\end{equation}
for some positive integers $a,b$ and $\xi_n\in\bb{R}_{>0}$. Putting $\bb{L}_n$ in \eqref{eq:stima} with $S$ the set of all places $v_n$ of $\bb{L}_n$ over $\mathfrak{p}$, for all $\alpha\in\bb{L}_n$ that are not roots of unity and such that $\alpha^a\neq \sigma_n(\alpha)^b$ we have
\begin{equation*}
    h(\alpha)\geq \frac{1}{a+b}\left(\frac{[\bb{K}_v:\bb{Q}_p]}{[\bb{K}:\bb{Q}]}\xi_n\log p-\log 2\right).
\end{equation*}
To have a bound on every element of $\bb{L}_n$ we need to produce for each $\gamma\in\bb{L}_n$ an element $\beta\in\bb{L}_n$ such that $h(\gamma)\geq C h(\beta)$, for some constant $C>0$ independent of $\gamma$, and such that $\beta^a\neq \sigma_n(\beta)^b$. For this we will choose $\beta=\tau(\gamma)/\gamma^{g}$ for some special automorphism $\tau\in\Gal(\bb{L}/\bb{K})$ and integer $g>1$. The element $\beta$ will also have the property that for all conjugates $\beta^\prime$ of $\beta$ over $\bb{K}$ that are different from $\beta$, $\beta^\prime/\beta$ will never be a $p$-power root of unity.  Applying Lemma \ref{lemma:adzs} we deduce a bound on $h(\beta)$ and therefore on $h(\alpha)$ that depends on $n$. To eliminate this dependency first notice that if we put $a=b=1$, equation \eqref{eqn:stima} says that $\sigma_n$ belongs to some ramification group for the places $v_n$ of $\bb{L}_n$ above $\mathfrak{p}$; moreover the condition $\beta^a\neq\sigma_n(\beta)^b$ becomes $\beta\neq\sigma_n(\beta)$. It is therefore natural to take each $\bb{L}_n$ to be the field fixed by the normal closure of some ramification group inside $\Gal(\bb{L}_{n+1}/\bb{K})$. Taking $\sigma_n$ in the last nontrivial ramification group $H_n$ of $\bb{L}_n/\bb{K}$ we have the bound
\begin{equation*}
    |\sigma_n(\gamma) - \gamma|_v \leq p^{-\frac{(i_n + 1)}{ee_n}}
\end{equation*}
for all $\gamma\in\mathcal{O}_{\bb{L}_n}$, where $e_n$ is the ramification index of $\mathfrak{p}$ in $\bb{L}_n$ and $i_n$ the filtration index of $H_n$. If the quantities ${(i_n + 1)}/{e_n}$ and $|\overline{H_n}|$ are bounded by constants independent of $n$, we can use the following \say{acceleration} lemma (this is also why we need that $\beta^\prime/\beta$ is never a $p$-power root of unity).
\begin{lemma}[Lemma 2.1, \cite{ADZ14}]\label{lemma:accel}
    Let $\bb{K}$ be a number field, $v$ a finite place of $\bb{K}$ over $p$ and $\xi\in\bb{R}_{>0}$. Let $\gamma_1,\gamma_2\in\mathcal{O}_{\bb{K}}$ be such that $|\gamma_1-\gamma_2|_v\leq p^{-\xi}$. Then for any non-negative integer $\lambda$ we have $|\gamma_1^{p^\lambda}-\gamma_2^{p^\lambda}|_v\leq p^{-s_{p,\xi}(\lambda)}$ with $s_{p,\xi}(\lambda)\rightarrow +\infty$ as $\lambda\rightarrow +\infty$.
\end{lemma}
Using Lemma \ref{lemma:accel} we find a big enough $\lambda>0$ such that
\begin{equation*}
    |\sigma_n(\gamma^{p^\lambda}) - \gamma^{p^\lambda}|_v \leq p^{-|\overline{H_n}|}
\end{equation*}
for all $n$ at once. Taking $\phi\in\Gal(\bb{L}_n/\bb{Q})$, since $|\phi^{-1}(x)|_v=|x|_{\phi(v)}$ the previous bound will work for all the places in the orbit of $v$ under the action of any $\phi$, commuting with $\sigma_n$. Choosing $S$ as the orbit of this action and bounding its size, we apply Lemma \ref{lemma:adzs} again to deduce
\begin{equation*}
    h(\beta) \geq \frac{\log(p/2)}{2p^{\lambda}}.
\end{equation*}
\subsection{Theorem \ref{thm:probinf} and the merging theorem}\label{subs:merging}
We now formalize strategy described in the previous section. Recall that we have fixed a prime $p$ satisfying assumption \ref{ass1} and an embedding $\iota:\overline{\Q}\hookrightarrow\overline{\Q}_p$. Let $\mathfrak{p}$ be the ideal of $\bb{K}$ determined by $\iota$ (see the Notations subsection) and $e$ its absolute ramification index. Consider a general Galois extension $\mathbb{L}$ of a number field $\K$ and the following assumptions on $\bb{L}$.
\begin{assumption}\label{ass2}
    There exists $\tau \in G_{\K} \cap I$, where $I$ is an inertia group over $\mathfrak{p}$,
    such that
    \begin{equation*}
        \varepsilon_p(\tau)=g \text { with } g \in \mathbb{Z}, g>1
    \end{equation*}
    and $\tau_{\mid\mathbb{L}} \in Z(\Gal(\mathbb{L}/\K))$ (the element $\tau$ is called a \emph{central element}).
\end{assumption}
\begin{assumption}\label{ass3}
    There exist real constants $C_1, C_2 \geq 1$ and a sequence of distinct finite Galois extensions $\left(\mathbb{L}_n /\K\right)_{n \geq 0}$
    \begin{equation*}
        \K=\mathbb{L}_0 \subset \mathbb{L}_1 \subset \cdots \subset \mathbb{L}_n \subset \cdots \subset \mathbb{L}, \quad \text { and } \quad \mathbb{L}=\bigcup_n \mathbb{L}_n
    \end{equation*}
    with the following properties. For $n \geq 1$ the extension $\mathbb{L}_n /\K$ is ramified over $\mathfrak{p}$ with ramification index $e_n$. Let $\left(G_i^n\right)_i$ be the sequence of ramification groups of $\mathbb{L}_n /\K$ above $\mathfrak{p}$. Denote by $H_n=G_{i_n}^n$ the last nontrivial ramification group and $\overline{H_n}$ its normal closure in $\Gal\left(\mathbb{L}_n /\K\right)$, then
    \begin{equation*}
        \frac{e_n}{i_n+1} \leq C_1, \quad|\overline{H_n}| \leq C_2
    \end{equation*}
    and
    \begin{equation*}
        \mathbb{L}_n^{\overline{H_n}}=\mathbb{L}_{n-1}.
    \end{equation*}
\end{assumption}
The following proposition tells us that if $\bb{L}$ satisfies \hyperref[ass2]{assumption 2}, for every $\alpha\in \bb{L}$ we can find, using the central element $\tau$, another element $\beta$ that has smaller height than $\alpha$. Moreover $\beta$ has the special property that $\sigma(\beta)/\beta$ is never a root of unity of order a power of $p$ for any $\sigma\in\Gal(\bb{L}/\bb{K})$; this will be used to apply the acceleration lemma \ref{lemma:accel}. If $\bb{L}$ also satisfies \hyperref[ass3]{assumption 3} we can find a $\sigma\in\Gal(\bb{L}/\bb{K})$ and a conjugate $\beta^\prime$ of $\beta$ such that $\beta^\prime\neq \sigma(\beta^\prime)$. We want this to apply the product formula on $\sigma(\beta^\prime)-\beta^\prime$ that need to be nonzero.
\begin{proposition}\label{prop:specialset}
    Let $\bb{L}$ be a Galois extension of a number field $\bb{K}$ and let $\alpha\in\bb{L}^\times$ that is not a root of unity.
    \begin{enumerate}
        \item\label{item:one} If $\bb{L}/\bb{K}$ satisfies \hyperref[ass2]{assumption 2}, the element $\beta=\tau(\alpha)/\alpha^g$ satisfies 
            \begin{equation*}
                h(\alpha)\geq h(\beta)/(g+1).
            \end{equation*}
            Moreover for all $\sigma\in\Gal(\bb{L}/\bb{K})$ that do not fix $\beta$, the element $\sigma(\beta)/\beta$ is not a root of unity of order a power of $p$.
        \item Suppose that $\bb{L}/\bb{K}$ satisfies \hyperref[ass2]{assumption 2} and \hyperref[ass3]{assumption 3}. Let $\beta$ be as in (\ref{item:one}) and suppose that $\beta$ does not belong to $\bb{K}$. Then there exists an element $\beta^\prime\in\bb{L}$ conjugate to $\beta$ over $\bb{K}$, that is not fixed by ${H_n}$ i.e. there exists $\sigma\in{H_n}$ such that $\sigma(\beta^\prime)\neq\beta^\prime$.
    \end{enumerate}
\end{proposition}
\begin{proof}
    We prove this point by point.
    \begin{enumerate}
        \item If $\beta$ were a root of unity
        \begin{equation}\label{eq:radici2}
            h(\alpha)=h(\tau(\alpha))=h(\beta\alpha^g)=h(\alpha^g)=gh(\alpha)
        \end{equation}
        and since $g>1$ equation \eqref{eq:radici2} would imply that $h(\alpha)=0$ which is a contradiction since $\alpha$ is neither 0 nor a root of unity. Observe also that
        \begin{equation*}
            h(\beta)=h(\tau(\alpha)\alpha^{-g})\leq h(\tau(\alpha))+gh(\alpha)=(g+1)h(\alpha).
        \end{equation*}
        so $h(\alpha)\geq h(\beta)/(g+1)$. Since $\beta\notin\bb{K}$ there exists a $\sigma\in\Gal(\bb{L}/\bb{K}$ such that $\sigma(\beta)\neq \beta$. By definition $\tau$ is central so
        \begin{equation*}
            \frac{\sigma(\beta)}{\beta}=\frac{{\sigma}\tau(\alpha)}{\tau(\alpha)}\left(\frac{{\sigma}(\alpha^{ g})}{\alpha^{ g}} \right)^{-1}=\frac{\tau(\eta)}{\eta^g}
        \end{equation*}
        where $\eta={{\sigma}(\alpha)}/{\alpha}\in\bb{L}$. Now if ${\sigma(\beta)}/{\beta}$ was a $p$-power root of unity, its height would be zero and since $g>1$, $\eta$ would be a root of unity for the same reason as \eqref{eq:radici2}. Writing $\eta=\eta_1\eta_2$ with $\eta_1$ a $p$-root of unity and $\eta_2$ with order not divisible by $p$ we have $\tau(\eta_1)/\tau(\eta_1^g)=1$ since $\varepsilon_p(\tau)=g$. This implies that $\sigma(\beta)/\beta=\tau(\eta_2)/\tau(\eta_2^g)$ has order not divisible by $p$ which is a contradiction, so ${\sigma(\beta)}/{\beta}$ is not a $p$-power root of unity.
        \item Let $n>0$ be the smallest integer such that $\beta\in\bb{L}_n$. Since $\bb{L}_n^{\overline{H_n}}=\bb{L}_{n-1}$ and $\beta\notin\bb{L}_{n-1} $, there exist $\varphi\in\Gal(\bb{L}/\bb{K})$ and $\sigma\in H_n$ such that $\varphi^{-1}\sigma\varphi (\beta)\neq\beta$. Putting $\beta^\prime:=\varphi(\beta)$ we have $\sigma(\beta^\prime)\neq \beta^\prime$ which proves the proposition.
        \end{enumerate}
\end{proof}
We can now prove that \hyperref[ass2]{assumption 2} and \hyperref[ass3]{assumption 3} imply property (B).
\begin{theorem}\label{thm:probinf}
    Suppose that $\bb{L}/\bb{K}$ satisfies \hyperref[ass2]{assumption 2} and \hyperref[ass3]{assumption 3}, then $\bb{L}$ has property (B).
\end{theorem}
\begin{proof}
    Let $\alpha$ be a nonzero element of $\bb{L}$ that is not a root of unity and let $\beta=\tau(\alpha)/\alpha^g$ as in Proposition \ref{prop:specialset}. Suppose that $\beta$ does not belong to $\bb{K}$; by Proposition \ref{prop:specialset} there exists an element $\beta^\prime\in \bb{L}$ conjugate to $\beta$ over $\bb{K}$, and $\sigma\in H_n$ such that $\sigma(\beta^\prime)\neq\beta^\prime$. Recall that 
    \begin{equation*}
        h(\alpha)\geq h(\beta)/(g+1)=h(\beta^\prime)/(g+1)
    \end{equation*}
    so to bound $h(\alpha)$ we can bound $h(\beta^\prime)$. Since $ \sigma \in H_n$, for any $ \gamma \in \mathcal{O}_{\mathbb{L}_n} $ we have
    \begin{equation*}
        |\sigma(\gamma) - \gamma|_v \leq p^{-\frac{(i_n + 1)}{ee_n}}
    \end{equation*}
    where $e_n$ is the ramification degree of a place of ${\mathbb{L}_n}$ above $\mathfrak{p}$. By Lemma \ref{lemma:accel}, there exists a integer $\lambda>0$ such that
    \begin{equation}\label{eqn:diseq}
        |\sigma(\gamma^{p^\lambda}) - \gamma^{p^\lambda}|_v \leq p^{-|\overline{H_n}|}.
    \end{equation}
    Let $\phi$ be an element in the centralizer $C(\sigma)$ of $\sigma$ in $\Gal(\bb{L}_n/\bb{K})$. By exchanging $\gamma$ with $\phi^{-1}(\gamma)$ in \eqref{eqn:diseq}, since $ \sigma \phi^{-1} = \phi^{-1} \sigma $ and $ |\phi^{-1}(x)|_v = |x|_{\phi(v)}$, we get
    \begin{equation*}
        |\sigma(\gamma^{p^\lambda}) - \gamma^{p^\lambda}|_w \leq p^{-|\overline{H_n}|}
    \end{equation*}
    with $ w = \phi(v)$. The same inequality is true for any other place in the orbit $S$ of $v$ under the action of $ C(\sigma) $ on the places above $\mathfrak{p}$. For any $p$-power $p^\lambda$, we have $ \sigma(\beta^\prime)^{p^\lambda} \neq (\beta^\prime)^{p^\lambda}$ by Proposition \ref{prop:specialset}. Applying Lemma \ref{lemma:adzs} with $S$ as above we have the bound
    \begin{equation*}
        h(\beta^\prime)\geq\frac{1}{2p^\lambda}\left(|S|\frac{[\bb{L}_{n,v} : \bb{K}_w]}{[ \bb{L}_n : \bb{K}]}|\overline{H_n}|\log p-\log 2\right).
    \end{equation*}
    We now bound $|S|$. Since $\sigma$ belongs to $\overline{H}_n$ that is normal in $\Gal(\bb{L}_n/\bb{K})$, th orbit $O$ of $\sigma$ under the action of conjugation is contained in $\overline{H_n}$ and
    \begin{equation*}
        |C(\sigma)| = \frac{[ \bb{L}_n : \bb{K}]}{|O|} \geq \frac{[ \bb{L}_n : \bb{K}]}{|\overline{H_n}|}.
    \end{equation*}
    The stabiliser of $v$ under the action of $\Gal(\bb{L}_n/\bb{K})$ is by definition the decomposition group $D$ of $v$ over $\mathfrak{p}$. Thus,
    \begin{equation*}
        |S| = \frac{|C(\sigma)|}{|D\cap C(\sigma)|} \geq \frac{[ \bb{L}_n : \bb{K}]}{[\bb{L}_{n,v} : \bb{K}_v]} |\overline{H_n}|^{-1},
    \end{equation*}
    and we get
    \begin{equation*}
    h(\beta^\prime) \geq\frac{\log(p/2)}{2p^{\lambda}} > 0.
    \end{equation*}
    Suppose that $\beta\in\bb{K}$. Since $\bb{K}$ is a number field it has property (B) so $h(\beta)>c$ for some real positive constant. Since $h(\alpha)\geq h(\beta)/(g+1)$ by Proposition \ref{prop:specialset}, we have $h(\alpha)>c/(g+1)$. Taking $C=\min\{{\log(p/2)}/{2p^{\lambda}},c/(g+1)\}$ we have
    \begin{equation*}
        h(\alpha)\geq C>0.
    \end{equation*}
\end{proof}
Theorem \ref{thm:probinf} shows that any field satisfying \hyperref[ass2]{assumption 2} and \hyperref[ass3]{assumption 3} has property (B), but we can say more. Since \hyperref[ass3]{assumption 3} concerns the ramification of $\bb{L}$ above $\mathfrak{p}$, nothing should change if we take the compositum of $\bb{L}$ with a finite extension of $\bb{F}$ of $\bb{K}$ that is \emph{unramified} at $\mathfrak{p}$. Similarly, since the element $\tau$ of \hyperref[ass2]{assumption 2} lies in an inertia group above $\mathfrak{p}$, it should be a central element for $\bb{FL}$ as well. Indeed the following theorem shows that assumptions \hyperref[ass2]{2} and \hyperref[ass3]{3} are preserved for finite extensions unramified at $\mathfrak{p}$. Now assuming that $\bb{F}/\bb{K}$ is an \emph{infinite} extension, unramified at $\mathfrak{p}$ with property (B), we can break $\bb{F}$ into a union of finite unramified extensions with uniform bounded height and take successive composita with $\bb{L}$. Since $\bb{F}$ has property (B), the height of each compositum is bounded uniformly, so $\bb{FL}$ has property (B).
\begin{theorem}\label{thm:composto}
        Let $\K$ be a number field, $\mathbb{F} /\K$ and $\mathbb{L} /\K$ two Galois extensions and suppose that:
    \begin{enumerate}
        \item The field $\mathbb{F}$ is unramified at $\mathfrak{p}$ and satisfies property (B).
        \item The field $\bb{L}$ satisfies assumptions \ref{ass2} and \ref{ass3}.
    \end{enumerate}
    Then the compositum $\mathbb{FL}$ has property (B).
\end{theorem}
\begin{proof}
    Write $\bb{F}$ as a union of finite Galois extensions $\bb{F}_m/\bb{K}$ with $\bb{F}_m\subseteq\bb{F}_{m+1}$ and fix $m$. The restriction $\tau\mid_{\bb{F}_m}$ belongs to the inertia group $I$ of $\bb{F}_m$ over $\mathfrak{p}$, which is trivial since $\bb{F}_m$ is unramified at $\mathfrak{p}$. Moreover since $\tau\vert_{\bb{L}}\in Z(\Gal(\bb{L}/\bb{K}))$ we have 
    \begin{equation*}
        \tau\vert_{\bb{F}_m\bb{L}}\in Z(\Gal(\bb{F}_m\bb{L}/\bb{K}))\cap\Gal(\bb{F}_m\bb{L}/\bb{F}_m)\subseteq Z(\Gal(\bb{F}_m\bb{L}/\bb{F}_m))
    \end{equation*}
    We conclude that the extension $\bb{F}_m\bb{L}/\bb{F}_m$ satisfies \hyperref[ass2]{assumption 2}. We have that
    \begin{equation*}
    \Gal(\bb{F}_m\bb{L}/\bb{F}_m)\simeq\Gal(\bb{L}/\bb{F}_m\cap\bb{L})\hookrightarrow\Gal(\bb{L}/\bb{K}).
    \end{equation*}
    Denote by $F_m,L$ and $K$ the completions of $\bb{F}_m, \bb{L}$ and $\bb{K}$ with respect to the valuation $v$ induced by the embedding $\iota$. Since ${F}_m/K$ is unramified the inertia group of $L/F_m\cap L$ and $L/K$ coincide and so do the higher ramification groups. Using $\iota$ we identify the higher ramifications group and the inertia groups as subgroups of $\Gal(\bb{L}/\bb{F}_m\cap \bb{L})$ which is normal in $\Gal(\bb{L}/\bb{K})$. We conclude that the normal closure of $H_n$ in $\Gal(\bb{L}/\bb{F}_m\cap\bb{L})$ is contained in the normal closure in $\Gal(\bb{L}/\bb{K})$. Since every element of $H_n$ acts trivially on $\bb{F}_m$, we have $(\bb{F}_m\bb{L})^{\overline{H_n}}=\bb{F}_m\bb{L}^{\overline{H_n}}=\bb{F}_m\bb{L}_{n-1}$ so $\bb{F}_m\bb{L}/\bb{F}_m$ satisfies \hyperref[ass3]{assumption 3}.  The field $\bb{F}$ has property (B), so there exists a constant $c>0$ such that for all $m$, every nonzero element of $\bb{F}_m$ that is not a root of unity has height at least $c$. Every element of $\bb{F}_m\bb{L}$ has height bounded from below by $C=\min\{{\log(p/2)}/{2p^{\lambda}},c\}$. Since $\lambda$ and $c$ do not depend on $m$, the bound holds for all elements of $\bb{FL}$, hence $\bb{FL}$ has property (B).
\end{proof}
\begin{remark}
    In Theorem \ref{thm:probinf} the assumption that $\bb{K}$ is a number field is only used to apply the Bogomolov property and have a height lower bound on the elements of $\bb{K}$. Since in this paper $\bb{K}$ is a number field, and both assumptions \hyperref[ass2]{2} and \hyperref[ass3]{3} are stated with $\bb{K}$ a number field, we have decided to structure the proof of Theorem \ref{thm:composto} by passing to finite extension of $\bb{F}_m$. Both Theorem \ref{thm:probinf} and Theorem \ref{thm:composto} are still true if we replace $\bb{K}$ with a field with property (B).
\end{remark}
\subsection{Outside of \texorpdfstring{$p$}{}}
Now we study the part \say{outside of $p$} of $\widehat{\rho}$. We start with the following definition which was inspired by \cite{ADZ14}.
\begin{definition}\label{def:adz}
    Let $\bb{K}$ be a number field and $\bb{L}/\bb{K}$ a Galois extension of $\bb{K}$ with Galois group $G$. We say that $\bb{L}$ has the \emph{ADZ property} or that it is a \emph{ADZ field} if the extension $\bb{L}^{Z(G)}/\bb{K}$ has bounded local degrees over some non Archimedean place $v$ of $\bb{K}$. We will say that a field has ADZ or that is an ADZ field \emph{at $v$} to specify the place.
\end{definition}
This definition stems from one of the main results of \cite{ADZ14} where it is shown that property ADZ implies property (B) (see \cite[Theorem 1.5]{ADZ14}). In this section we will prove that being ADZ is preserved under compositum and therefore exhibit a new class of fields with property (B) that is closed under compositum i.e. taking the compositum also has property (B). We first need a useful lemma on the center of the fiber product of groups.
\begin{lemma}\label{lemma:centro}
    Let $I$ be set of indices and $(G_i)_{i\in I}$ be a family of groups. Let $\{\phi_i:G_i\rightarrow H\}_{i\in I}$ a family of surjective morphisms onto some fixed group $H$. If
    \begin{equation*}
        G=G_1\times_HG_2\times_H\cdots =\left\{(g_i)_{i\in I}\in \prod_{i\in I}G_i\mid \phi_i(g_i)=\phi_j(g_j)\text{ for all } i,j\in I\right\}
    \end{equation*}
    denotes the fiber product over $H$, then
    \begin{equation*}
        Z(G)=\{(g_i)_{i\in I}\in G\mid g_i\in Z(G_i) \text{ for all $i\in I$}\}.
    \end{equation*}
\end{lemma}
\begin{proof}
    Let $\pi_i : G\rightarrow G_i$ be the projection on the $i$-th factor. If an element $g\in G$ is such that $\pi_i(g)\in Z(G_i)$ for all $i\in I$, then it clearly lies in the center of $G$. On the other hand let $g\in Z(G)$. First observe that since each $\phi_i$ is surjective onto $H$, so is each projection map $\pi_i$. To show that $\pi_i(g)\in Z(G_i)$ we will show that it commutes with any $h_i\in G_i$. Since each projection map is onto, we can find an element $g^\prime \in G$ such that $ \pi_i(g ^\prime ) = h_i$. We then have $g^\prime \cdot g \cdot (g^\prime) ^{-1} = g$ since $g \in Z(G)$ and applying the map $\pi_i$ to this equality gives $h_i\pi_i(g)h^{-1}_i = \pi_i(g)$. This shows that $h_i\in Z(G_i)$ which concludes the proof.
\end{proof}
We now state a lemma about the field fixed by the center of a compositum of fields. Roughly it states that \say{the field fixed by the center of the compositum is the compositum of the fields fixed by the centers}. 
\begin{lemma}\label{lemma:centrocomposito}
    Let $I$ be set of indices and $\{\bb{L}_i\}_{i\in I}$ be a family of Galois extensions of a number field $\bb{K}$ lying in a common algebraic closure. Let $G_i=\Gal(\bb{L}_i/\bb{K})$. Denote by $\bb{L}$ the compositum of all the $\bb{L}_i$ and $G$ its Galois group over $\bb{K}$. We have that $\bb{L}^{Z(G)}$ equals the compositum of the fixed fields $\bb{L}_i^{Z(G_i)}$.
\end{lemma}
\begin{proof}
    Let $H_i=\Gal(\bb{L}/\bb{L}_i)$, so that $H_i$ is a normal subgroup of $G$ with $G/H_i=G_i$. The restriction map $\pi_i : G\twoheadrightarrow G_i$ induces the isomorphism
    \begin{equation*}
        \frac{G}{H_iZ(G)}\simeq\frac{G_i}{Z(G_i)}.
    \end{equation*}
    This means that the subfield of $\bb{L}$ fixed by $H_iZ(G)$ coincides with the subfield of $\bb{L}_i$ fixed by $Z(G_i)$. Since $H_iZ(G)=\pi_i^{-1}(Z(G))$ we have shown that
    \begin{equation*}
        \bb{L}^{\pi_i^{-1}(Z(G))}=\bb{L}_i^{Z(G_i)}.
    \end{equation*}
    The compositum of the $\bb{L}^{\pi_i^{-1}(Z(G_i))}$ is the smallest field that contains them, and the corresponding subgroup of $G$ is the intersection of all the $\pi_i^{-1}(Z(G_i))$. We have
    \begin{equation*}
        \bigcap_{i\in I} \pi_i^{-1}(Z(G_i))=\{(g_j)\in G\mid g_i\in Z(G_i)\text{ for all } i\in I\}=Z(G)
    \end{equation*}
    where the last inequality comes from Lemma \ref{lemma:centro}.
\end{proof}
It is easy now to show that the compositum of ADZ fields is again ADZ.
\begin{theorem}\label{thm:compostoadz}
    Let $I$ be a set of indices and $\bb{K}$ a number field. Let $\{\bb{L}_i\}_{i\in I}$ be a family of Galois extensions of $\bb{K}$ that are all ADZ at a place $v$ of $\bb{K}$. If each $\bb{L}_i$ has bounded local degree uniformly in $i$, then the compositum of all the $\bb{L}_i$ is an ADZ field at $v$.
\end{theorem}
\begin{proof}
    Let $\bb{L}$ be the compositum of all the $\bb{L}_i$. Put $G=\Gal(\bb{L}/\bb{K})$ and $G_i=\Gal(\bb{L}_i/\bb{K})$ for each $i\in I$. By Lemma \ref{lemma:centrocomposito} the field $\bb{L}^{Z(G)}$ is the compositum of all the $\bb{L}_i^{Z(G_i)}$. Let $w$ be a place on $\bb{L}^{Z(G)}$ extending $v$, then the completion of $\bb{L}_w$ equals the compositum of each completion. By \cite[Proposition 14, chapter II]{Lan94} there are only finitely many field extensions of bounded degree of $\bb{K}_v$. This means that the completion of $\bb{L}^{Z(G)}$ at $w$ is actually the compositum of finitely many fields with bounded local degree, so it also has bounded local degree.
\end{proof}
With Theorem \ref{thm:compostoadz} we can easily prove that the representation $\rho^{(p)}$ has property (B) when restricted to the absolute Galois group of an ADZ field at $p$.
\begin{proposition}\label{prop:adz}
    Let $\bb{K}/\bb{Q}$ be an ADZ field at $p$, then $\bb{K}(\rho_{}^{(p)})$ has property (B).
\end{proposition}
\begin{proof}
    The field $\bb{K}(\rho^{(p)})$ is the compositum of $\bb{K}$ and $\bb{Q}(\rho^{(p)})$; since $\bb{K}$ is ADZ at $p$, by Theorem \ref{thm:compostoadz} it is sufficient to show that $\bb{Q}(\rho^{(p)})$ is ADZ at $p$. Since every ADZ field has property (B) by \cite[Theorem 1.5]{ADZ14}, the result will follow. Take $\mathfrak{Q}$ a prime of $\overline{\bb{Q}}$ above $p$, then the characteristic polynomial of $\rho^{(p)}({\frob_\mathfrak{Q}})$ is
    $$X^2+a_p(f)X+\chi(p)p^{k-1}.$$
    Since $a_p(f)=0$ by \hyperref[ass1]{assumption 1}, applying the Cayley-Hamilton theorem we deduce that $\rho^{(p)}({\frob_\mathfrak{Q}^2})$ is scalar and therefore central in $\Gal(\bb{\Q}(\rho^{(p)})/\bb{Q})$. Letting $Z=Z(\Gal(\bb{\Q}(\rho^{(p)})/\bb{Q}))$, the order of the Frobenius element is at most 2 in $\bb{\Q}(\rho^{(p)})^Z$. Since the residual degree of $\bb{Q}(\rho^{(p)})^Z$ is the order of the Frobenius element we see that $\bb{\Q}(\rho^{(p)})$ is ADZ at $p$.
\end{proof}
\begin{remark}\label{rem:adz}
     Using Theorem \ref{thm:compostoadz}, if $\bb{F}/\bb{Q}$ is an ADZ field at $p$ the field $\bb{F}\bb{Q}^{ab}$ is ADZ (since any abelian extension is trivially ADZ at every prime and has property (B). As remarked in the introduction, property (B) has been established in \cite{AD00} for the maximal abelian extension of $\bb{Q}$ and subsequently for the maximal abelian extension of number fields in \cite{AZ10}. In this direction one could therefore ask whether the maximal abelian extension of an ADZ field has property (B), but this is not true as the following argument due to A. Plessis (private communication) shows. Since $\bb{Q}^{ab}$ is abelian, it clearly has the ADZ property. Consider now the extension $\bb{Q}^{ab}(2^{1/n})$ where $2^{1/n}$ denotes an $n$-th root of of 2. This extension is abelian since its Galois group is isomorphic to a subgroup of $\Gal(\bb{Q}(\zeta_n,2^{1/n})/\bb{Q}(\zeta_n))$ that is cyclic by Kummer theory. As a consequence we see that for all $n$, the element $2^{1/n}$ belongs to the maximal abelian extension of $\bb{Q}^{ab}$. This shows that the maximal abelian extension of $\Q^{ab}$ cannot have the Bogomolov property. Observe that one can make the same argument by considering any ADZ field $\bb{K}/\bb{Q}$ in place of $\bb{Q}$: the extension $\bb{K}\bb{Q}^{ab}$ is again ADZ and taking $n$-th roots produces abelian extensions by Kummer theory. Again this shows that the maximal abelian extension of $\bb{K}\bb{Q}^{ab}$ cannot have (B). 
\end{remark}
As a corollary we deduce that $\bb{K}(\rho^{(p)})$ has property (B) for every number field $\bb{K}$. This concludes the study of $\rho^{(p)}$. To deal with $\rho_{p}$ we will rely on the description of Breuil of some special crystalline representations.
\subsection{Crystalline representations with character}
Recall that in subsection \ref{subs:merging} we have fixed a prime $p$ satisfying assumption \ref{ass1}, an embedding $\iota:\overline{\Q}\hookrightarrow\overline{\Q}_p$ with $\mathfrak{P}$ the associated ideal and $v$ the place associated to this embedding. By restricting the local representation $\rho_{v}$ to $G_{\Q_p}$ one obtains a two dimensional representation of $G_{\Q_p}$. In \cite[Theorem 1.2.4 (ii)]{Sch90}, \citeauthor{Sch90} proves the following result, that we reproduce in streamlined form without mentioning motives.
\begin{theorem}[\cite{Sch90}, Theorem 1.2.4 (ii)]\label{thm:scholl}
    Let $f$ be a normalised newform of weight $k$, of level $N\geq3$ and nebentypus $\chi$. Denote by $\bb{E}$ the coefficient field of $f$ and suppose that $p$ does not divide $N$ and $p\geq k$. Let $\pi$ be a prime of $\bb{E}$ above $p$, then the $\pi$-adic representation of $\Gal(\Qb_p/\Q_p)$ associated to $f$ is crystalline and the characteristic polynomial of $\varphi$ on the associated $\varphi$-filtered module is equal to
    \begin{equation*}
        X^2-a_p(f)X+\chi(p)p^{k-1}.
    \end{equation*}
\end{theorem}
 As Scholl remarks shortly after the theorem, the condition $p\geq k$ can be removed by the work of Faltings \cite{Fal89}. In \cite[Section 4.3.2]{BM02}, the authors show that if the restriction of $\rho_v$ to $G_{\Q_p}$ is reducible, then $a_p(f)$ must be a $p$-adic unit; since $a_p(f)=0$ by \hyperref[ass1]{assumption 1}, we see that $\rho_v\vert_{ G_{\Q_p}}$ must be irreducible. Recall that we also have assumed $N\geq 3$ since $\chi$ is nontrivial. 
 Two dimensional irreducible crystalline representations of $G_{\Q_p}$ have been classified by \citeauthor{Bre03} in \cite[Proposition 3.1.1]{Bre03} and are of the form $V_{k^\prime,a,\varpi}:=V_{k^\prime,a}\otimes\varpi$ where 
 \begin{itemize}
    \item $k^\prime\geq 2$ is an integer;
    \item $a$ is an element of $\mathfrak{P}$;
    \item $V_{k^\prime,a}=\operatorname{Hom}_{\varphi,\operatorname{Fil}}(D_{k^\prime,a},B_{cris})$ where $D_{k^\prime,a}$ is a 2-dimensional weakly admissible filtered $\varphi$-module (see \cite[Section 3.1]{Bre03} for the precise definition), and $B_{cris}$ is Fontaine's ring of crystalline periods;
    \item $\varpi$ is a character
    \begin{equation*}
        \varpi: G_{\Q_p}\rightarrow\mathcal{O}_{\Qb_p}^\times
    \end{equation*}
    that is the product of an unramified character by an integral power of the cyclotomic character at $p$ (a so called \emph{crystalline character}).  The Hodge-Tate weight of $\varpi$ is precisely the exponent of the cyclotomic character.
 \end{itemize}
\begin{remark}
    Given a local field $K$, there is (see \cite[Theorem 1.1]{Mil20} a unique continuous homomorphism (called the \emph{Artin map}) 
    \begin{equation}\label{eqn:artinmap}
        \theta_K:K^\times\rightarrow \Gal(K^{ab}/K)
    \end{equation}
    with the following properties. For all abelian extensions $L/K$, denoting $\theta_{L/K}$ the composition of $\theta_K$ with the restriction $\Gal(K^{ab}/K)\twoheadrightarrow\Gal(L/K)$ we have
\begin{itemize}
\item if $L/K$ is unramified, then the image of a uniformizer $\pi$ of $L$ is the Frobenius element of $L$;
\item $\theta_{L/K}$ is surjective with kernel the norm group $N_{L/K}(L^\times)$, inducing the isomorphism $K^\times/N_{L/K}(L^\times)\simeq \Gal(L/K)$.
\end{itemize}
Moreover, by the local existence theorem \cite[Theorem 1.4]{Mil20}, the Artin map induces a canonical isomorphism of profinite groups
\begin{equation}\label{eqn:artiniso}
    \widehat{\theta}_K:\widehat{K^\times}\rightarrow \Gal(K^{ab}/K)
\end{equation}
sending $\mathcal{O}_K^\times$ isomorphically onto the absolute inertia group of $K^{ab}$. The notation $\varpi(p)$ is a shorthand to denote $\varpi(\theta_K(p))$
\end{remark}
The representations $V_{k^\prime,a,\varpi}$ have Hodge-Tate weight $(0,k^\prime-1)$ and the characteristic polynomial of $\varphi$ is $X^2-aX+\varpi(p^{-1})^2p^{k^\prime-1}$. From the definition of $\varpi$ we see that the Hodge-Tate weight of $\varpi$ is zero if and only if it is unramified. On the other hand, the theorem of Scholl tells us that the characteristic polynomial of $\varphi$ is $X^2-a_pX+\chi(p)p^{k-1}$; by \cite{Fal87} representations associated to modular forms of weight $k$ have Hodge-Tate weight $(0, k-1)$ so we conclude that
\begin{equation}\label{eqn:nebentypus}
    k^\prime=k,\quad a=a_p,\quad \chi(p)=\varpi(p^{-1})^2
\end{equation}
for all $p\nmid N$. By \hyperref[ass1]{assumption 1}, $a_p=0$ so $a=0$ and by the explicit description of $V_{k,0}$ in \cite[Proposition  3.1.2]{Bre03} we obtain the following
\begin{proposition}\label{lemma:breuil}
    Let $q=p^2$ and denote by $\bb{Q}_q$ the unique unramified quadratic extension of $\bb{Q}_p$. For every place $v$ of $\bb{E}$ above $p$, the restriction of $\rho_{v}$ to $G_{\Q_q}$ is isomorphic to the crystalline representation
    \begin{equation*}
        V_{k,0,\varpi}\simeq\left(\operatorname{Ind}^{G_{\Q_p}}_{G_{\Q_{q}}}\epsilon_2^{k-1} \right)\otimes\mu_{\sqrt{-1}}\otimes \varpi
    \end{equation*}
    where:
    \begin{itemize}
        \item $\epsilon_2 : G_{\Q_q} \rightarrow \Z_q^\times$ is one of the two continuous characters characterized by the property that their composition with $\Q^\times_q\rightarrow G^{ab}_{\Q_q}$ (given by local class field theory) is trivial on p and is one of the two natural embeddings when restricted to $\Z^\times_q$;
        \item $\mu_{\sqrt{-1}}$ is one of the two unramified characters of $G_{\Q_p}$ sending $\frob_p$ to $\sqrt{-1}$.
    \end{itemize}
\end{proposition}
\noindent Moreover \cite[Proposition 3.1.1]{Bre03} says that the only isomorphisms between these representations are $V_{k,a_p,\varpi}\simeq V_{k,-a_p,\varpi\mu_{\sqrt{-1}}}$, so if $a_p=0$ the only isomorphisms are $V_{k,0,\varpi}\simeq V_{k,0,\varpi\mu_{\sqrt{-1}}}$. We now give some properties of $\epsilon_2$ and $\varpi$.
\begin{remark}\label{rem:local}
    The character $\epsilon_2$ is defined using the following diagram.
    \begin{equation*}
        \begin{tikzcd}
            {G_{\mathbb{Q}_q}} & {G_{\mathbb{Q}_q}^{ab}} & {\widehat{\mathbb{Q}_q^\times}\simeq \mathbb{Z}_q^\times\times\widehat{\mathbb{Z}}} \\
            && {\mathbb{Z}_q^\times}
            \arrow["{\operatorname{Res}}",two heads, from=1-1, to=1-2]
            \arrow["{\epsilon_2}"', from=1-1, to=2-3]
            \arrow["{\widehat{\theta}^{-1}_{\Q_q}}", from=1-2, to=1-3]
            \arrow[from=1-3, to=2-3]
        \end{tikzcd}
    \end{equation*}
   Letting $\tau$ be the nontrivial automorphism of $\Gal(\bb{Q}_q/\bb{Q}_p)$, composing $\epsilon_2$ with $\tau$, we get the other character $\epsilon_2^\prime$. Since $\epsilon_2$ has abelian image, it factors through the maximal abelian extension of $\Q_q$ and is still surjective when restricted to the inertia group $I\subseteq G_{\bb{Q}_q}$. Since the inertia group of $\Q_q(\epsilon_2)/\Q_q$ is isomorphic to $\epsilon_2(I)$, we see that $\Q_q(\epsilon_2)/\Q_q$ is totally ramified.
\end{remark}
Observe that by Mackey's restriction formula
\begin{equation*}
    \operatorname{Res}^{G_{\Q_p}}_{G_{\Q_{q}}}\left(\operatorname{Ind}^{G_{\Q_p}}_{G_{\Q_{q}}}\epsilon_2^{k-1}\right)=\epsilon_2^{k-1}\oplus \epsilon_2^{\prime \ k-1}.
\end{equation*}
We have
\begin{equation*}
    {V_{k,0,\varpi}}_{\mid G_{\Q_q}}\simeq (\epsilon_2^{k-1}\oplus \epsilon_2^{\prime \ k-1})_{\mid G_{\Q_p}}\otimes {\mu_{\sqrt{-1}}}_{\mid G_{\Q_q}}\otimes \varpi_{\mid G_{\Q_q}}.
\end{equation*}
Since both ${V_{k,0}}$ and ${V_{k,0}}\otimes \varpi$ have Hodge-Tate weights $\{0,k-1\}$ we conclude that $\varpi$ has Hodge-Tate weight 0 so it is unramified. This implies that $\varpi$ factors through $\Q_p^{ur}$
(the maximal unramified extension of $\Q_p$) so it is uniquely determined by its value at $\frob_{\mathfrak{P}}$. Since $\theta_{\Q_q}(p)=\frob_{\mathfrak{P}}$ the last relation in \eqref{eqn:nebentypus} shows that $\varpi$ has finite image, since $\chi$ is a Dirichlet character.
\begin{remark}\label{rem:conj}
    Recall that denoting $\tau$ the nontrivial element of $\Gal(\bb{Q}_q/\bb{Q}_p)$, we have $\epsilon_2^\prime=\tau(\epsilon_2)$, therefore
    \begin{equation*}
        \epsilon_2\cdot\epsilon_2^\prime=N_{\Q_q/\Q_p}(\epsilon_2)
    \end{equation*}
    where $N_{\Q_q/\Q_p}$ is the field norm.
\end{remark}
The image of the representation $\rho_{v}$ is a priori a subset of $\gl_2(\mathcal{O}_{\bb{E}_v})$ but one can do better. In fact a result of Carayol states that one can always \say{descend} to a representation which is equivalent, after extension of scalars, to a representation in the ring the traces. This will be used in the proof of the normal closure lemma.
\begin{proposition}\label{prop:car}
    Under \hyperref[ass1]{assumption 1}, up to conjugation by an element $\operatorname{GL}_2(\mathcal{O}_{\bb{E}_v})$ the representation $\rho_{v}$ restricted to $G_{\Q_p}$ has image in $\operatorname{GL}_2(R)$ where $R$ is the finite extension of $\Z_p$ obtained by adjoining the image of $\varpi$.
\end{proposition}
\begin{proof}
Recall that $\rho_{v{\mid G_{\Q_p}}}$ is isomorphic to $V_{k,0,\varpi}$. By \cite[Lemma 2.1]{CT21} the residual representation of $\rho_{v}\vert_{G_{\Q_p}}$ is irreducible since $p+1\nmid k+1$ so we deduce that $V_{k,0}$ is absolutely irreducible. We can then apply \cite[Theorem 2]{Car94} and find an element $A\in\gl_2(\mathcal{O}_{\bb{E}_v})$ such that $A V_{k,0} A^{-1}$ has image in $\gl_2(\Z_p)$ since the traces lie in $\Z_p$. Since $V_{k,0,\varpi}=V_{k,0}\otimes \varpi$ and $\varpi$ has abelian image, we have $AV_{k,0,\varpi}A^{-1}=AV_{k,0}A^{-1}\otimes \varpi$ which proves the claim.
\end{proof}
\begin{remark}\label{rem:img}
    Throughout the rest of the paper we will denote by ${\rho}_{f,v\mid G_{\Q_p}}$ the representation given by Proposition \ref{prop:car}. This will have image in $\gl_2(R)$ and denoting with $\varpi^*$ the dual of $\varpi$, $\rho_{v\mid G_{\Q_p}}\otimes \varpi^*\simeq V_{k,0}$ will have image in $\gl_2(\Z_p)$.
\end{remark}
Let $\bb{K}$ be the abelian extension of $\bb{Q}$ fixed by the intersection of the kernels of all the Dirichlet characters associated to the inner twists of $f$ (see the paragraph about inner twists of Section \ref{sec3}). Let $\bb{K}_{v}$ be the completion of $\bb{K}$ at $v$.
\begin{lemma}\label{lemma:fiss}
    The field $\bb{K}_v$ contains $\Q_p(\varpi)$.
\end{lemma}
\begin{proof}
    By Remark \ref{rmk:neben} we have $G_\bb{K}\subseteq \ker \chi$ so, after restricting $\chi$ to $G_{\bb{Q}_p}$, it sufficient to prove that $\ker \chi \subseteq \ker\varpi$. Observe that since $\chi$ and $\varpi$ are unramified at $p$, they are determined by the value at $\frob_\mathfrak{P}$. Now using the last equality of \eqref{eqn:nebentypus} we see that
    \begin{equation*}
        \chi(\frob_\mathfrak{P})\varpi(\frob_\mathfrak{P})^2=1
    \end{equation*}
    so we conclude that $\chi\varpi^2$ is the constant character with value 1. Restricting to $\ker\chi$ we have $\varpi^2=1$ so up to choosing $\varpi\mu_{\sqrt{-1}}$, that yields the isomorphic representation $V_{k,0,\varpi\mu_{\sqrt{-1}}}$ (see right after Proposition \ref{lemma:breuil}), we have that $\varpi$ is trivial on $\ker\chi$. We conclude that $\ker\varpi\supseteq \ker\chi$.
\end{proof}
Since $\bb{K}_v$ is unramified the compositum ${L^\prime}=\bb{K}_v\Q_q$ is again a finite and unramified extension of $\Q_q$ and by Lemma \ref{lemma:fiss} the restriction of $\rho_{v}$ to $G_L$ kills the crystalline character i.e.
\begin{equation*}\label{eqn:kill}
    \rho_v\vert_{G_{L^\prime}}\simeq{V_{k,0,\varpi}}\vert_{G_{L^\prime}}\simeq (\epsilon_2^{k-1}\oplus \epsilon_2^{\prime \ k-1})\vert_{G_{L^\prime}}\otimes {\mu_{\sqrt{-1}}}\vert_{G_{L^\prime}}.
\end{equation*}
In \cite{AT25} the weight $k$ was assumed even since they deal with forms in $S_k(\Gamma_0(N))$. They use this fact twice:
\begin{enumerate}
    \item Before \cite[Lemma 4.6]{AT25} where they say that since $k$ is even
    \begin{equation*}
        \frac{q-1}{(q-1,k-1)}\neq 1.
    \end{equation*}
    However the only way to have $q-1=(q-1,k-1)$ is that $q-1$ divides $k-1$ which is impossible since we have assumed that $\frac{p+1}{2}$ does not divide $k-1$.
    \item In the proof of \cite[Proposition 4.7]{AT25} where they observe that since $k$ is even, $-1$ belongs in the set of elements of $\bb{Z}_q$ that are $k-1$ powers. This is used to \say{absorb} the image of $\mu_{\sqrt{-1}}\vert_{G_{\bb{Q}_q}}$ into the image of $\epsilon_2^{k-1}$.
\end{enumerate}
Since in our case $k$ is not necessarily even, we have to get rid of the image of $\mu_{\sqrt{-1}}$ so we restrict to the unramified quadratic extension $L$ of $L^\prime$ where $\mu_{\sqrt{-1}}$ is trivial.
The extension $L/\Q_q$ is unramified while $\Q_q(\epsilon_2^{k-1})/\Q_q$ is totally ramified (see \cite[Lemma 4.6]{AT25}).
By \cite[Lemma 2.1,(iii)]{Hab13}, for each $i\geq -1$ the map
\begin{equation}\label{eqn:hab}
    \Gal(L(\epsilon_2^{k-1})/L)\cap G_i(L(\epsilon_2^{k-1}) /\Q_q)\rightarrow G_i(\Q_q(\epsilon_2^{k-1})/\Q_q)
\end{equation}
is an isomorphism. Since $L/\Q_q$ is unramified all the ramification of $L(\epsilon_2^{k-1})/\Q_q$ is \say{above} $L$ meaning that $G_i(L(\epsilon_2^{k-1})/\Q_q)=G_i(L(\epsilon_2^{k-1})/L)$ for all $i\geq 0$. Using \eqref{eqn:hab} we have an isomorphism
\begin{equation}\label{eqn:ramif}
    G_i(L(\epsilon_2^{k-1})/L)\simeq G_i(\Q_q(\epsilon_2^{k-1})/\Q_q).
\end{equation}
We now introduce some notation. For any ideal $I$ of $\mathcal{O}_{\bb{K}_v}$ define $\rho_v(I)$ as the composition 
\begin{equation*}
    G_\Q\rightarrow \gl_2(\mathcal{O}_{\bb{K}_v})\rightarrow \gl_2(\mathcal{O}_{\bb{K}}/I).
\end{equation*}
If $I=(\alpha)$ we will denote $\rho_v(I)$ by $\rho_v(\alpha)$. Assume that $p\mathcal{O}_\bb{E}=\prod_{v\vert p}\mathfrak{p}_v^{e_v}$, then 
\begin{equation}\label{eqn:rhopn}
    \rho_p(p^n)=\prod_{v\vert p}\rho_{v}(\mathfrak{p}_v^{e_n n}):G_\Q\rightarrow \prod_{v\vert p} \gl_2(\mathcal{O}_{\bb{E}_v}/\mathfrak{p}_v^{e_n n})=\gl_2(\mathcal{O}_\bb{E}/p^n\mathcal{O}_\bb{E}).
\end{equation}
Denote by $L(p^n)$ the field $L(\rho_p(p^n))$.
\begin{remark}\label{rem:pinf}
    Let $\mathfrak{p}$ be the nontrivial prime ideal of $\mathcal{O}_{\bb{E}_v}$ and denote by $\bb{Q}(\mathfrak{p}^\infty)$ the compositum of all $\Q(\mathfrak{p}^n)$ inside $\bb{Q}(\rho_{v})$ as $n$ varies. Since $\bigcap_{n>0}\mathfrak{p}^n$ equals the identity by Krull's intersection theorem, we have
        \begin{equation}
        \ker(\rho_v)=\rho_v^{-1}\left(\bigcap_{n>0}\mathfrak{p}^n\right)=\bigcap_{n>0}\ker\rho_v(\mathfrak{p}^n).
    \end{equation}
    This shows that $G_{\bb{Q}(\rho_v)}=G_{\bb{Q}(\mathfrak{p}^\infty)}$ or equivalently $\bb{Q}(\rho_v)=\bb{Q}(\mathfrak{p}^\infty)$. Similarly, denoting $\bb{Q}({p}^\infty)$ the compositum of all the $\bb{Q}(p^n)$ as $n$ varies, one can show that $\bb{Q}(\rho_p)=\bb{Q}({p}^\infty)$.
\end{remark}
We have now the analogue of \cite[Proposition 4.7]{AT25} in the case of nontrivial nebentypus.
\begin{proposition}\label{Prop:gal}
Let $n \geq 1$ and $L$ as above.
\begin{enumerate}
    \item[(i)] Let $\delta=\frac{q-1}{(q-1,k-1)}$. The extension $L(p^n)/L$ is totally ramified and abelian of degree $\delta q^{n-1}$. Moreover we have
    \begin{equation*}
        \Gal\left(L(p^n)/L\right) \simeq \Z/\delta\Z \times \left(\Z/p^{n-1}\Z\right)^2.
    \end{equation*}
    \item[(ii)] Let $j$ and $i$ be integers with $1 \leq j \leq n$ and $q^{j-1} \leq (q-1,k-1)i \leq q^j - 1$. The higher ramification groups are given by
    \begin{equation*}
        G_i(L(p^n)/L) = \Gal(L(p^n)/L(p^j)).
    \end{equation*}
    In particular, the last nontrivial higher ramification group has index
    \begin{equation*}
        i_n := \frac{q^{n-1} - 1}{(q-1,k-1)}
    \end{equation*}
    and
    \begin{equation*}
        G_{i_n} = 
        \begin{cases}
            (\Z/p\Z)^2 & \text{if } n \geq 2 \\
            \Z/\delta\Z & \text{if } n = 1.
        \end{cases}
    \end{equation*}
    \item[(iii)] If $M$ is an integer prime to $p$, for all places $w\vert p$ of $\mathcal{O}_\bb{E}$, the image of $\rho_{w}\vert_{G_{L}}$ contains the scalar matrix $\begin{pmatrix} M^{k-1} & 0 \\ 0 & M^{k-1} \end{pmatrix}$. Therefore the image of the representation $\rho_p\vert_{G_L}$ contains the element $\begin{pmatrix} M^{k-1} & 0 \\ 0 & M^{k-1} \end{pmatrix}$.
\end{enumerate}
\end{proposition}
\begin{proof}
By equation \eqref{eqn:kill} we have
\begin{equation}\label{eqn:mackey}
    \rho_v\vert_{G_L} \simeq (\epsilon_2^{k-1}\oplus \epsilon_2^{\prime \ k-1})\vert_{G_L}.
\end{equation}
The proof is now identical to the one in \cite[Proposition 4.7]{AT25} after replacing $L$ with $\Q_q$ and \eqref{eqn:ramif}
\end{proof}
\subsection{The proof of Theorem \ref{thm:main}}
In this section we shall prove the main Theorem \ref{thm:main}. Define $\bb{K}(p^n)=\bb{K}(\rho_p(p^n))$ and $L(p^n)=L(\rho_p(p^n))$ where $\rho_p(p^n)$ was defined by \eqref{eqn:rhopn}. The last result we need is the following
\begin{proposition}[Normal closure lemma]\label{prop:ncl}
    The normal closure of the Galois group $\Gal(L(p^n)/L(p^{n-1}))$ in $\Gal(\K(p^n)/\K)$ is $\Gal(\K(p^n)/\K(p^{n-1}))$.
\end{proposition}
\noindent It will be proven in Section \ref{sec5}. Assuming the normal closure lemma one can show that $\bb{K}(\rho_p)$ satisfies assumptions \ref{ass2} and \ref{ass3} as follows.
\begin{enumerate}
    \item To construct a central element one starts with the scalar element $\tau$ inside $\Gal(L(\rho_{v})/L)$ given by the last point of Proposition \ref{Prop:gal}. One then lifts $\tau$ to an inertia group and takes an appropriate power to have image in $\Z$ under the cyclotomic character $\varepsilon_p$. The result is an automorphism satisfying \hyperref[ass2]{assumption 2}.
    \item For \hyperref[ass3]{assumption 3} one takes $\bb{L}_n=\bb{K}(p^n)$; Proposition \ref{Prop:gal} gives the desired inequalities on the size of the ramification groups and the ramification indices. By the normal closure lemma $\overline{\Gal(L(p^n)/L(p^{n-1}))}=\Gal(\bb{L}_n/\bb{L}_{n-1})$; since the last non trivial ramification group $H_n$ of $\bb{L}(p^n)/\bb{L}(p^{n-1})$ is precisely $\Gal(L(p^n)/L(p^{n-1}))$ one has that $\bb{L}_n^{\overline{H_n}}=\bb{L}_{n-1}$. Therefore \hyperref[ass3]{assumption 3} is met.
\end{enumerate}
We give now rigorous proofs of points (1) and (2).
\begin{proposition}\label{prop:cep}
    The field $\bb{K}(\rho_{p})$ has the central element property.
\end{proposition}
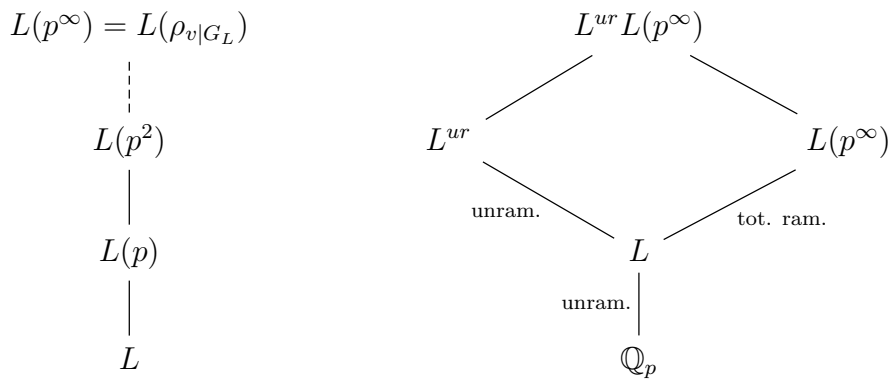
\begin{figure}\label{fig:estensioni}
    \centering
    \begin{tikzcd}
        {L(p^\infty)=L(\rho_{v})} &&& {L^{ur}(p^\infty)} \\
        {L(p^2)} && {L^{ur}} && {L(p^\infty)} \\
        {L(p)} &&& L \\
        L &&& {\mathbb{Q}_p}
        \arrow[no head, from=1-4, to=2-5]
        \arrow[dashed, no head, from=2-1, to=1-1]
        \arrow[no head, from=2-3, to=1-4]
        \arrow[no head, from=3-1, to=2-1]
        \arrow[no head, from=3-1, to=4-1]
        \arrow["{\text{unram.}}", no head, from=3-4, to=2-3]
        \arrow["{\text{tot. ram.}}"', no head, from=3-4, to=2-5]
        \arrow["{\text{unram.}}", no head, from=4-4, to=3-4]
    \end{tikzcd}
    \caption{Field extensions considered in the proof of Proposition \ref{prop:cep}.}
\end{figure}
\begin{proof}
By Proposition \ref{Prop:gal} with $M=2$ there exists an element $\tau\in\Gal(L(\rho_{p})/L))$ corresponding to the scalar matrix $\begin{pmatrix} 2^{k-1} & 0 \\ 0 & 2^{k-1} \end{pmatrix}$. Since $L/K$ is unramified and $K(\rho_p)/K$ is totally ramified, they are linearly disjoint; moreover $L(\rho_p)/L$ is totally ramified so we have
\begin{equation*}
    \Gal(L^{ur}(\rho_p)/L^{ur})\simeq\Gal(L(\rho_p)/L)\simeq\Gal(K(\rho_p)/K).
\end{equation*}
We can lift $\tau\in\Gal(L(\rho_p)/L)$ to an element (that we still call $\tau$) of the inertia subgroup of $G_{\Q_p}$ such that its restriction to $K(\rho_p)$ is central. We have
\begin{equation*}
    \det\rho_{v}(\tau)=\varepsilon_p(\tau)^{k-1}=2^{2(k-1)}
\end{equation*}
so $\varepsilon_p(\tau)=4\zeta$ for some root of unity $\zeta\in\Z_p^\times$. Since $\Z_p^\times$ contains only $(p-1)$-roots of unity, by replacing $\tau$ with $\tau^{p-1}$ we get rid of $\zeta$ and we have found a central element.

\end{proof}
\begin{proposition}\label{prop:ass3}
    Assuming Proposition \ref{prop:ncl}, the field $\bb{K}(\rho_{p})$ satisfies \hyperref[ass3]{assumption 3}.
\end{proposition}
\begin{proof}
    It is sufficient to prove our result for $\bb{K}(p^\infty)$ by Remark \ref{rem:pinf}. Let $n\geq 1$ consider the filtration given by $\mathbb{L}_{n} = \K(p^{n})$. By the first two points of Proposition \ref{Prop:gal}, using the notation of \hyperref[ass3]{assumption 3} we have (recall that $\delta=\frac{q-1}{(q-1,k-1)}$ from Proposition \ref{Prop:gal})
    \begin{equation*}
        {e_{n} = \delta q^{n-1} = \frac{(q-1)q^{n-1}}{(q-1,k-1)}, \quad i_{n} = \frac{q^{n-1}-1}{(q-1,k-1)},}
        \quad H_{n} \simeq \operatorname{Gal}(L(p^{n})/L(p^{n-1})).
    \end{equation*}
    Thus
    \begin{equation*}
        \frac{e_{n}}{i_{n}+1} = \frac{(q-1)q^{n-1}}{q^{n-1}-1 + (q-1,k-1)} \leq q-1
    \end{equation*}
    and by Proposition \ref{prop:ncl} we have $\overline{H_{n}} = \operatorname{Gal}(\mathbb{L}_{n}/\mathbb{L}_{n-1})$ and $\mathbb{L}_{n}^{\overline{H_{n}}} = \mathbb{L}_{n-1}$. In the proof of Proposition \ref{prop:ncl} one shows that  $\overline{H_n}$ embeds into $\gl_2(\mathcal{O}_{\bb{E}}/p\mathcal{O}_{\bb{E}})$ for $n=1$ and into $M_2(\mathcal{O}_{\bb{E}}/p\mathcal{O}_{\bb{E}})$ for $n\geq 2$ so taking the norm we have the bound
    \begin{equation*}
        |\overline{H_{n}}| \leq p^{4[\bb{E}:\mathbb{Q}]}
    \end{equation*}
    \hyperref[ass3]{Assumption 3} therefore holds with $C_{1} = p^{2}$ and $C_{2} = p^{4[\bb{E}:\mathbb{Q}]}$.
\end{proof}
We can finally prove Theorem \ref{thm:main}.
\begin{proof}[Proof (of Theorem \ref{thm:main})]
    We want to show that the representation $\widehat{\rho}:G_\Q\rightarrow \gl_2(\widehat{\mathcal{O}_\bb{E}})$ has property (B) or in other words that $\Q(\widehat{\rho})$ has property (B). Recall that $\bb{K}$ is the abelian extension of $\bb{Q}$ fixed by all the Dirichlet characters associated to the inner twists of $f$.
    The field $\mathbb{F} = \K(\rho^{(p)})$ is unramified at $p$ and by Propositions \ref{prop:cep} and \ref{prop:ass3} the field $\mathbb{L} = \bb{K}(\rho_{p})$ satisfies \hyperref[ass2]{assumption 2} and \hyperref[ass2]{assumption 3}. By Theorem \ref{thm:composto} the composite $\bb{FL}$ has property (B) but
    \begin{equation*}
        \mathbb{F}\bb{L}=\K(\rho^{(p)}\times\rho_p)=\K(\widehat{\rho})
    \end{equation*}
    so we have shown that $\widehat{\rho}$ has property (B). Since 
    $\Q(\widehat{\rho})\subseteq \K(\widehat{\rho})$ we have proved Theorem \ref{thm:main}.
\end{proof}
\noindent Note that we have actually proved something stronger i.e. that the field $\K(\widehat{\rho})$ has property (B). Observe also that given a normalized eigenform $f\in S_k(\Gamma_0(N),\chi)$ and a prime $p$ satisfying assumption \hyperref[ass1]{assumption 1}, there always exists a newform $g\in S_k(\Gamma_0(M),\chi)$ of level $M$ dividing $N$ such that the adelic representations of $f$ and $g$ are isomorphic. Moreover since $p\nmid N$, $a_p(f)=a_p(g)=0$ so $g$ satisfies the hypotheses of Theorem \ref{thm:main}. Since the adelic representations are isomorphic we see that the adelic representation associated to $f$ also has property (B).
\section{The proof of the normal closure lemma}\label{sec5}
To fully prove Theorem \ref{thm:main} we have to provide a proof of Proposition \ref{prop:ncl} of which we now give a sketch. For the rest of the section we put $\mathcal{A}=\mathcal{O}_{\bb{F}}\otimes_\Z \Z_p$ as in Remark \ref{rem:algebra} and 
\begin{gather*}
    \widehat{G}(\mathcal{A}):=\{\gamma\in\gl_2(\mathcal{A})\mid \det(\gamma)=x^{k-1},x\in\Z_p^\times\},\\
    \widehat{G}(\mathcal{O}_\bb{F}/p\mathcal{O}_\bb{F}):=\{\gamma\in\gl_2(\mathcal{O}_\bb{F}/p\mathcal{O}_\bb{F})\mid \det(\gamma)=x^{k-1},x\in\bb{F}_p^\times\}.
\end{gather*}
The proof is divided in two cases.
\begin{itemize}
    \item For $n=1$ one uses the explicit description of the Galois groups and the ramification groups of $L(p)/L$ in Proposition \ref{Prop:gal} and the group theoretic result \cite[Proposition A.1.1]{AT25}. The latter says that if a subgroup of $\gl_2(\bb{F}_p)$ is big enough and its image under the determinant map is also big, then its normal closure inside $\widehat{G}(\mathcal{O}_\bb{F}/p\mathcal{O}_\bb{F})$ is $\widehat{G}(\mathcal{O}_\bb{F}/p\mathcal{O}_\bb{F})$ itself.
    \item For $n>1$ the idea is to define an injective \say{log} map
    \begin{equation*}
        \mathcal{L}_2:\Gal(\bb{K}(p^n)/\bb{K}(p^{n-1}))\rightarrow M_2(\mathcal{O}_{\bb{F}}/p\mathcal{O}_{\bb{F}})
    \end{equation*}
    that linearizes the problem. One then shows that both $\overline{\Gal(L(p^n)/L(p^{n-1}))}$ and $\Gal(\bb{K}(p^n)/\bb{K})$ are mapped to the set 
    \begin{equation*}
        \widehat{M_2}(\mathcal{O}_{\bb{F}}/p\mathcal{O}_{\bb{F}})=\{X\in M_2(\mathcal{O}_{\bb{F}}/p\mathcal{O}_{\bb{F}})\mid \tr (X)\in\bb{F}_p\}.
    \end{equation*}
     Denoting $M_2(\mathcal{O}_{\bb{F}}/p\mathcal{O}_{\bb{F}})^0$ the trace zero matrices one applies \cite[Proposition A.2.5]{AT25} with $V=\mathcal{L}_2(\overline{\Gal(L(p^n)/L(p^{n-1}))})\cap {M}_{2}(\mathcal{O}_{\bb{F}}/p\mathcal{O}_{\bb{F}})^{0}$ to get $$\mathcal{L}_2(\overline{\Gal(L(p^n)/L(p^{n-1}))})\supseteq {M}_{2}(\mathcal{O}_{\bb{F}}/p\mathcal{O}_{\bb{F}})^{0}.$$ Using that $\rho_v\vert_{G_L}\simeq V_{k,0}$ one has
    $$\mathcal{L}_2(\overline{\Gal(L(p^n)/L(p^{n-1}))})\supseteq\mathbb{F}_p\cdot{I}.$$ 
    Since $\widehat{{M}}_{2}(\mathcal{O}_{\bb{F}}/p\mathcal{O}_{\bb{F}}) = {M}_{2}(\mathcal{O}_{\bb{F}}/p\mathcal{O}_{\bb{F}})^{0} \oplus \mathbb{F}_p\cdot{I}$ the result follows.
\end{itemize}
\begin{lemma}\label{lemma:cicl}
    The cyclotomic character is surjective onto $\bb{Z}_p^\times$ when restricted to $G_L$.
\end{lemma}
\begin{proof}
    We know that $\varepsilon_p:G_\bb{Q}\rightarrow \bb{Z}_p^\times$ is surjective and totally ramified at $p$. Since $L/\bb{Q}_p$ is unramified, while $\bb{Q}_p(\varepsilon_p)$ is totally ramified over $\bb{Q}_p$, we have
    \begin{equation*}
        G_LG_{\Q_p(\varepsilon_p)}=G_{L\cap \Q_p(\varepsilon_p)}=G_{\Q_p};
    \end{equation*}
    since $\ker\varepsilon_p=G_{\Q_p(\varepsilon_p)}$ we have 
    $$\varepsilon_p(G_{\Q_p})=\varepsilon_p( G_LG_{\Q_p(\varepsilon_p)})=\varepsilon_p(G_L)$$
    so $\varepsilon_p$ is still surjective when restricted to $G_L$. 
\end{proof}
Recall the following results of Ribet
\begin{proposition}[\cite{Rib75} Corollary 2.2]\label{prop:rib}
    Let $K_1,\ldots, K_t$ be finite extensions of $\bb{Q}_p$ with $p\geq 5$ and let $\mathcal{O}_1,\ldots,\mathcal{O}_t$ be their ring of integers. Denoting the $k-1$ powers of elements of $\mathbb{Z}_{p}^\times$ by $(\mathbb{Z}_{p}^{*})^{k-1}$, suppose that $\mathcal{G}$ is a closed subgroup of
    \begin{equation*}
        \mathcal{A}= \left\{(u_1,\ldots,u_t)\in\prod_{i=1}^{t}\gl_2(\mathcal{O}_i):\det u_1=\cdots=\det u_t\in(\mathbb{Z}_{p}^{*})^{k-1}\right\}
    \end{equation*}
    which satisfies:
    \begin{enumerate}
        \item[(i)] $\det\mathcal{G}=(\mathbb{Z}_{p}^{*})^{k-1}$;
        \item[(ii)] \textit{the image of $\mathcal{G}$ mod $p$ contains}
        $\prod_{i=1}^{t}\sl_2(\mathcal{O}_i/p \mathcal{O}_i)$.
    \end{enumerate}
    Then $\mathcal{G}=\mathcal{A}$.
\end{proposition}
Since $p$ satisfies \hyperref[ass1]{assumption 1}, we can apply Corollary \ref{prop:rib} to the completion of $\bb{F}$ at the places above $p$ and conclude that the image of $\rho_p(G_\bb{K})\supseteq\sl_2(\mathcal{O}_\bb{F}\otimes\bb{Z}_p)$. Since $\varepsilon_p$ is surjective when restricted to $G_L$ and 
\begin{equation*}
    \det\rho_p\vert_{G_L} =\varepsilon_p^{k-1}\vert_{G_L}
\end{equation*}
we conclude that $$\det\mathcal{G}=(\mathbb{Z}_{p}^{\times})^{k-1}.$$ Applying Proposition \ref{prop:rib} we conclude that $\rho_p(G_\bb{K})=\widehat{G}(\mathcal{A})$.\\
We prove the normal closure lemma with $n=1$.
\begin{proposition}[Normal closure lemma $n=1$]\label{prop:ncl1}
    The normal closure of the Galois group $\Gal(L(p)/L)$ in $\Gal(\K(p)/\K)$ is $\Gal(\K(p)/\K)$.
\end{proposition}
\begin{proof}
    We closely follow the proof of \cite[Proposition 5.1.]{AT25}. Let $G_1\subseteq \gl_2(\mathcal{O}_{\bb{F}}/p\mathcal{O}_{\bb{F}})$ be the image of $\Gal(L(p)/L)$ under $\rho_{p}(p)$.
    By Remark \ref{rem:img}, the group $G_1$ is contained in $\gl_2(\mathbb{F}_p)$ and has cardinality $\delta:=\frac{q-1}{(q-1,k-1)}$ by Proposition \ref{Prop:gal}. Let $s=\frac{p-1}{(p-1,k-1)}$, then since $\frac{p+1}{2}\nmid k-1$ by \hyperref[ass1]{assumption 1} (P3), the biggest common divisor between $p+1$ and $k-1$ is strictly less than $\frac{p+1}{2}$. After multiplying by $s$ both sides of $\frac{p+1}{(p+1,k-1)}>2$ we get $\delta >2s$. Recall that 
    \begin{equation*}
        \det(\rho_{p}(p))=\varepsilon_p^{k-1}(\chi\circ \varepsilon_p)\bmod p\mathcal{A}
    \end{equation*}
    and by restricting to $G_{L}$ we kill the character so
    \begin{equation*}
        \det(\rho_{p}(p)\vert_{G_L})=\varepsilon_p^{k-1} \bmod p\mathcal{A}.
    \end{equation*}
    By Lemma \ref{lemma:cicl} $\varepsilon_p$ mod $p\mathcal{A}$ is surjective onto $\bb{F}_p^\times$ when restricted to $G_L$. We have shown that $G_1$ satisfies the hypotheses of \cite[Proposition A.1.1]{AT25} so denoting by $\overline{G_1}$ the closure of $G_1$ in $\widehat{G}(\mathcal{O}_{\bb{F}}/p\mathcal{O}_{\bb{F}})$ we have shown
    \begin{equation*}
        \overline{G_1}=\widehat{G}(\mathcal{O}_{\bb{F}}/p\mathcal{O}_{\bb{F}}).
    \end{equation*}
    Since $\rho_p(p)$ is an isomorphism between $\Gal(\bb{K}(p)/\bb{K})$ and {$\widehat{G}(\mathcal{O}_{\bb{F}}/p\mathcal{O}_{\bb{F}})$} by \hyperref[ass1]{assumption 1} (P2), taking the normal closures in $\widehat{G}(\mathcal{O}_{\bb{F}}/p\mathcal{O}_{\bb{F}})$ we have
    \begin{equation*}
        \rho_p(p)(\Gal(\bb{K}(p)/\bb{K}))=\overline{G_1}=\overline{\rho_p(p)(\Gal(L(p)/L))}=\rho_p(p)(\overline{\Gal(L(p)/L)}).
    \end{equation*}
    We conclude that $\Gal(\bb{K}(p)/\bb{K})=\overline{\Gal(L(p)/L)}$.\\
    \end{proof}
    Before doing the case $n>1$ we need to introduce the logarithm map. We start with the following simple lemma.
    \begin{lemma}\label{lemma:log}
        Fix $n>1$ and let $N$ be a nonzero positive integer. Let $$K=\ker(\gl_N(\mathcal{O}_{\bb{F}}/p^n\mathcal{O}_{\bb{F}}))\rightarrow \gl_N(\mathcal{O}_{\bb{F}}/p^{n-1}\mathcal{O}_{\bb{F}}).$$ Every element of $K$ can be written as $I+p^{n-1}A$ where $A$ is well defined modulo $ p\mathcal{A}.$
    \end{lemma}
    \begin{proof}
        Take $X\in K$ and consider the matrix $X-I\in M_N(\mathcal{O}_\bb{F}/p\mathcal{O}_\bb{F})$. By definition we have that the reduction mod $p^{n-1}\mathcal{A}$ equals the zero matrix so one can write $X-I=p^{n-1}A$ for some matrix $A\in M_N(\mathcal{O}_{\bb{F}}/p^n\mathcal{O}_{\bb{F}})$. The matrices $I+p^{n-1}A$ and $I+p^{n-1}A^\prime$ define the same matrix in $K$ if and only if $A\equiv A^\prime\bmod p\mathcal{A}$.
    \end{proof}
    The previous lemma says that the following map is well defined.
    \begin{definition}
        Fix $n>1$ and let $N$ and $K$ as in Lemma \ref{lemma:log}. The map
        \begin{align*}
            \mathcal{L}_N:K&\rightarrow M_N(\mathcal{O}_{\bb{F}}/p\mathcal{O}_{\bb{F}})\\
            I+p^{n-1}A&\mapsto A\bmod p\mathcal{A}
        \end{align*}
        is called the logarithm map. It is an isomorphism between $K$ and the additive group $M_N(\mathcal{O}_{\bb{F}}/p\mathcal{O}_{\bb{F}})$.
    \end{definition}
    We state now some easy lemmas in preparation of the proof of the normal closure lemma with $n>1$. 
    \begin{lemma}
        Fix $n>1$. Let $B\in\gl_N(\mathcal{O}_{\bb{F}}/p^n\mathcal{O}_{\bb{F}})$ and $X\in K$. Denote by $\overline{B}$ the reduction of $B$ modulo $p\mathcal{A}$. We have
        \begin{equation}\label{eqn:conj}
            \overline{B}\mathcal{L}_N(X)\overline{B^{-1}}=\mathcal{L}_N(BXB^{-1}).
        \end{equation}
    \end{lemma}
    \begin{proof}
        Write $X=I+p^{n-1}A$, then
        \begin{equation*}
            \mathcal{L}_N(BXB)=\mathcal{L}_N(I+BAB^{-1})=\overline{B}\cdot\overline{A}\cdot\overline{B^{-1}}=\overline{B}\cdot\mathcal{L}_N({X)}\cdot\overline{B^{-1}}.
        \end{equation*}
\end{proof}

\begin{lemma}\label{lemma:additive}
    For $n>1$, the group $\Gal(L(p^n)/L(p^{n-1}))$ is isomorphic to a subgroup of the additive group $\widehat{M}_2(\mathcal{O}_{\bb{F}}/p\mathcal{O}_{\bb{F}})$ of matrices with coefficients in $\bb{F}_p$ via the logarithm map $\mathcal{L}_2$. Moreover the trace map $\tr:M_2(\bb{F}_p)\rightarrow\bb{F}_p$ is surjective when restricted to $\mathcal{L}_2(\Gal(L(p^n)/L(p^{n-1})))$.
\end{lemma}
\begin{proof}
    By definition, the group $\Gal(L(p^n)/L)$ is isomorphic to the image of $G_L$ under the map
    \begin{equation*}
        \rho_p(p^n):G_\bb{Q}\rightarrow\gl_2(\mathcal{O}_{\bb{F}}/p^{n}\mathcal{O}_{\bb{F}}).
    \end{equation*}
    If $K=\ker(\gl_2(\mathcal{O}_{\bb{F}}/p^{n}\mathcal{O}_{\bb{F}})\rightarrow\gl_2(\mathcal{O}_{\bb{F}}/p^{n-1}\mathcal{O}_{\bb{F}}))$, then $\rho_p(p^n)$ induces $$\Gal(L(p^n)/L(p^{n-1}))\simeq \rho_p(p^n)(G_L)\cap K.$$
    It is then clear that the $\mathcal{L}_2$ maps $\rho_p(p^n)(G_L)\cap K$ isomorphically to a subgroup of $M_2(\mathcal{O}_{\bb{F}}/p\mathcal{O}_{\bb{F}})$. 
    We now show that the image of $\rho_p(p^n)(G_L)\cap K$ under the logarithm map is $\bb{F}_p$. Since for $2\times 2$ matrices
    \begin{equation*}
        \det(X-\lambda I)= \lambda^2-\tr(X)\lambda+\det(X),
    \end{equation*}
    putting $X=p^{n-1}A$ and $\lambda=-1$ yields
    \begin{align*}
        &\det(I+p^{n-1}A)=\\
        &=1+p^{n-1}\tr(A)+(p^{n-1})^2\det A\equiv 1+p^{n-1}\tr(A)\bmod p^n\mathcal{A}.
    \end{align*}
    This shows that
    \begin{equation*}
        \mathcal{L}_1(\det\rho_p(p^n))=\tr(\mathcal{L}_2(\rho_p(p^n))).
    \end{equation*}
    The image of $G_L$ under the inverse of the Artin map \eqref{eqn:artiniso} 
    is $N_{L/\bb{Q}_q}(L^\times)$ so
    \begin{equation*}
        \varepsilon_2(G_L)=N_{L/\bb{Q}_q}(L^\times)\cap\bb{Z}_q^\times=N_{L/\bb{Q}_q}(\mathcal{O}_L^\times).
    \end{equation*}
    By Remark \ref{rem:conj} we have
    \begin{equation*}
        \det(\rho_p(G_L))=N^{k-1}_{\bb{Q}_q/\bb{Q}_p}(\epsilon_2(G_L))=N^{k-1}_{\bb{Q}_q/\bb{Q}_p}(N_{L/\bb{Q}_q}(\mathcal{O}_L^\times))=N^{k-1}_{L/\bb{Q}_p}(\mathcal{O}_L^\times).
    \end{equation*}
    Since the norm map is surjective by \cite[Ch. V, §2, Prop. 3]{Ser13} and $p\nmid k-1$ by (P3), we have $N^{k-1}_{L/\bb{Q}_p}(\mathcal{O}_L^\times)=(\bb{Z}_p^\times)^{k-1}$ and $\det(\rho_p(p^n)(G_L))=(\bb{Z}/p^n\bb{Z})^{\times}$. This proves the claim.
\end{proof}
Lastly recall the following
\begin{proposition}[Proposition A.2.5, \cite{AT25}]\label{prop:inv}
    Let $p\geq 3$ and $\mathcal{A}$ be a finite product of commutative $\mathbb{F}_p$-algebras. Let $V$ be a $\mathbb{F}_p$-subspace of $M_2(\mathcal{A})^0$ such that $V\cap M_2(\mathbb{F}_p)^0$ is nontrivial and invariant by the adjoint action of $\operatorname{SL}_2(\mathcal{A})$. Then $V=M_2(\mathcal{A})^0$.
\end{proposition}
We are now ready to conclude the proof of the normal closure lemma.
\begin{proposition}[Normal closure lemma $n>1$]
    Let $n>1$. The normal closure of the Galois group $\Gal(L(p^n)/L(p^{n-1})$ in $\Gal(\K(p^n)/\K(p^{n-1})$ is $\Gal(\K(p^n)/\K(p^{n-1}))$.
\end{proposition}

\begin{proof}
Let $\widehat{M}_2(\mathcal{O}_{\bb{F}}/p\mathcal{O}_{\bb{F}})$ be the $2\times2$ matrices with trace in $\bb{F}_p$ and ${M}_{2}(\mathcal{O}_{\bb{F}}/p\mathcal{O}_{\bb{F}})^{0}$ the set of trace zero matrices. We have the decomposition
\begin{equation*}
    \widehat{{M}}_{2}(\mathcal{O}_{\bb{F}}/p\mathcal{O}_{\bb{F}}) = {M}_{2}(\mathcal{O}_{\bb{F}}/p\mathcal{O}_{\bb{F}})^{0} \oplus \mathbb{F}_p\cdot{I}.
\end{equation*}
By Remark \ref{rem:img}, $\mathcal{L}_{n}(\Gal(L(p^n)/L(p^{n-1})))\subseteq {M}_{2}(\mathbb{F}_p)$ so our goal is to show
\begin{equation*}
    \mathcal{L}_2(\overline{\Gal(L(p^n)/L(p^{n-1}))})=\widehat{{M}}_{2}(\mathcal{O}_{\bb{F}}/p\mathcal{O}_{\bb{F}}).
\end{equation*}
This will prove that $$\widehat{{M}}_{2}(\mathcal{O}_{\bb{F}}/p\mathcal{O}_{\bb{F}})=\mathcal{L}_2(\overline{\Gal(L(p^n)/L(p^{n-1})}))=\mathcal{L}_2(\Gal(\K(p^n)/\K(p^{n-1}))$$ and, by the injectivity of $\mathcal{L}_2$, that $\overline{\Gal(L(p^n)/L(p^{n-1}))}=\Gal(\K(p^n)/\K(p^{n-1})$. By \hyperref[ass1]{assumption 1} (P2) the horizontal map of diagram \ref{diagram} is an isomorphism, so the composition with the restriction map in the same diagram is a surjection.
\begin{figure}[H]
    \centering
    \begin{equation}\label{diagram}
        \begin{tikzcd}
        	{\Gal(\bb{K}(p^n)/\bb{K})} \\
        	{\Gal(\bb{K}(p)/\bb{K})} & {{\widehat{G}\left(\mathcal{O}_{\bb{F}}/p\mathcal{O}_{\bb{F}}\right)}}
        	\arrow["\operatorname{res}", two heads, from=1-1, to=2-1]
        	\arrow[from=1-1, two heads, to=2-2]
        	\arrow["{\rho_p(p)}"', from=2-1, to=2-2]
        \end{tikzcd}
    \end{equation}
\end{figure}
Consider the adjoint action of {$\widehat{G}(\mathcal{O}_{\bb{F}}/p\mathcal{O}_{\bb{F}})$} on $M_2(\bb{F}_p)$:
\begin{align*}
    {\widehat{G}(\mathcal{O}_{\bb{F}}/p\mathcal{O}_{\bb{F}})}\times {M}_{2}(\mathbb{F}_p)&\rightarrow {M}_{2}(\mathbb{F}_p)\\
    (X, M)&\mapsto X\cdot M \cdot X^{-1}.
\end{align*}
By \eqref{eqn:conj} the subgroup $\mathcal{L}_2(\overline{\Gal(L(p^n)/L(p^{n-1}))})$ of ${M}_{2}(\mathbb{F}_p)$ is invariant under this action. Since ${M}_{2}(\mathcal{O}_{\bb{F}}/p\mathcal{O}_{\bb{F}})^{0}$ is also invariant, we have that
\begin{equation*}
    V=\mathcal{L}_2(\overline{\Gal(L(p^n)/L(p^{n-1}))})\cap {M}_{2}(\mathcal{O}_{\bb{F}}/p\mathcal{O}_{\bb{F}})^{0}
\end{equation*}
is an $\mathbb{F}_p$-subspace of ${M}_{2}(\mathcal{O}_{\bb{F}}/p\mathcal{O}_{\bb{F}})^{0}$ invariant under the adjoint action of {$\widehat{G}(\mathcal{O}_{\bb{F}}/p\mathcal{O}_{\bb{F}})$}. {Since $\operatorname{SL}_{2}(\mathcal{O}_{\bb{F}}/p\mathcal{O}_{\bb{F}})\subseteq\widehat{G}(\mathcal{O}_{\bb{F}}/p\mathcal{O}_{\bb{F}})$}, the subspace $V$ is invariant under the action of $\operatorname{SL}_{2}(\mathcal{O}_{\bb{F}}/p\mathcal{O}_{\bb{F}})$. If
\begin{equation*}
    \mathcal{L}_{n}(\Gal(L(p^n)/L(p^{n-1})))\cap{M}_{2}(\mathbb{F}_p)^{0}={0},
\end{equation*}
we could identify $\mathcal{L}_{n}(\Gal(L(p^n)/L(p^{n-1})))$ with a subgroup of ${M}_{2}(\mathbb{F}_p)/{M}_{2}(\mathbb{F}_p)^{0}$ which is impossibile since
\begin{equation*}
    \vert{M}_{2}(\mathbb{F}_p)/{M}_{2}(\mathbb{F}_p)^{0}\vert=p
\end{equation*}
while
\begin{equation*}
    \vert\mathcal{L}_{n}(\Gal(L(p^n)/L(p^{n-1})))\vert=p^2
\end{equation*}
by Proposition \ref{Prop:gal}. In particular this implies that 
\begin{equation*}
    V\cap{M}_{2}(\mathbb{F}_p)^{0} = \mathcal{L}_{n}(\overline{\Gal(L(p^n)/L(p^{n-1}))})\cap{M}_{2}(\mathbb{F}_p)^{0} \neq 0.
\end{equation*}
\noindent By Proposition \ref{prop:inv}, $V = M_{2}(\mathcal{O}_{\bb{F}}/p\mathcal{O}_{\bb{F}})^{0}$ and thus 
\begin{equation*}
    \mathcal{L}_{n}(\overline{\Gal(L(p^n)/L(p^{n-1}))})\supseteq M_{2}(\mathcal{O}_{\bb{F}}/p\mathcal{O}_{\bb{F}})^{0}.
\end{equation*}
By Lemma \ref{lemma:additive} the trace from $\mathcal{L}_2(\overline{\Gal(L(p^n)/L(p^{n-1}))})$ to $\mathbb{F}_p$ is surjective so
\begin{equation*}
    \mathcal{L}_2(\overline{\Gal(L(p^n)/L(p^{n-1}))})\supseteq\bb{F}_p\cdot I.
\end{equation*}
We conclude that
\begin{equation*}
    \mathcal{L}_2(\overline{\Gal(L(p^n)/L(p^{n-1}))})=\widehat{M}_2(\mathcal{O}_{\bb{F}}/p\mathcal{O}_{\bb{F}})
\end{equation*}
as desired.
\end{proof}
\section{Examples}\label{sec6}
We now discuss some examples. The main difficulty is to find newforms that satisfy the image assumption (P2); the main tool that we use is an unpublished result of Mascot \cite{Mas22}. This result roughly states given a newform $f\in S_k(\Gamma_0(N),\chi)$, under some assumptions on the irreducibility of residual representation and on the conductor of $\chi$, either the image of the residual representation satisfies (P2) or there exists another newform $g$ of lower level with isomorphic residual representation. To check that such $g$ does not exist one must check many spaces of newforms of high weight which is computationally demanding. However as we will see, by adding the condition $a_p=0$ we only need to compute the space of newforms of weight 2 or $k$. Before discussing Mascot's result we give a brief account of Serre's conjectures on modular forms (now a theorem due to \citeauthor{KW09a} \cite{KW09a},\cite{KW09b}) that, predict the minimal level, weight and the nebentypus of the newforms that produce the same residual representation as $f$.\\
\\
{\bf Serre's conjectures on modular forms.} Consider an odd, continuous irreducible representation
\begin{equation*}
    \rho:G_\bb{Q}\rightarrow\gl(V)
\end{equation*}
with $V$ a $\overline{\bb{F}}_p$-vector space of dimension 2. In \cite{Ser87} Serre makes the conjecture that $\rho$ is isomorphic to the residual representation attached to a modular form $f$; moreover he also predicts the minimal possible weight and level of $f$, as well as its nebentypus as follows (see \cite{Ser87} for Serre's original article or \cite[Definition 1.7]{Edi97} for a beautiful exposition).

\begin{itemize}
\item (The minimal level.) We need the notion of the Artin conductor of a character; we follow \cite[Chapter VI]{Ser13}. Consider a finite Galois extension $L/K$ of a local fields with Galois group $G$ and $\mathfrak{p}$ the nonzero prime ideal of $L$. Let $\chi$ be the character of a finite dimensional complex representation of $G$ and let $G_i$ be the $i$-th ramification group of $L/K$ with $i\geq 0$. Define
\begin{equation*}
    f(\chi,\mathfrak{p})=\sum_{i=0}^\infty \frac{|G_i|}{|G_0|}\dim V/V_i
\end{equation*}
where 
\begin{equation*}
    V_i=\{v\in V\mid\chi(g_i)v=v\text{ for all } g_i\in G_i\}.
\end{equation*}
The (local) \emph{Artin conductor} of $\chi$ is the ideal $\mathfrak{p}^{f(\chi,\mathfrak{p})}$. In the case of a finite Galois extension of number fields $\bb{L}/\bb{K}$ with Galois group $\bb{G}$, one defines the (global) \emph{Artin conductor} of a character $\chi$ of $\bb{G}$ as
\begin{equation*}
    \mathfrak{f}(\chi,\bb{G})=\prod_{\mathfrak{p}}\mathfrak{p}^{f(\chi,\mathfrak{p})}
\end{equation*}
where we have restricted each $\chi$ to a decomposition group at $\mathfrak{p}$ as $\mathfrak{p}$ varies over the nonzero prime ideals of $\bb{L}$. The idea is that the Artin conductor measures the ramification of $\chi$ by looking at how much each ramification group acts trivially on $V$. For example one sees that $\chi$ is unramified at $\mathfrak{p}$ if and only if $f(\chi,\mathfrak{p})=0$. Returning to representations coming from modular forms, we know that given a newform $f\in S_k(\Gamma_0(N),\chi)$, for each prime $p$ there is a system of $p$-adic representations $\{\rho_v\}_v$ (see \hyperref[sec2]{Section 2})
\begin{equation*}
    \rho_v:G_\bb{Q}\rightarrow \gl_2(\mathcal{O}_{\bb{E}_v})
\end{equation*}
indexed by the places $v$ of its Hecke field $\bb{E}$ above $p$. Each representation $\rho_v$ is unramified outside of $p N$; this means that the level $N$ contains information about the \emph{ramification} of the representation outside of $p$. In view of these considerations, one might be tempted to define of the global Artin conductor for representations in characteristic $p$ and to define the minimal level $N$ as (a generator of) the Artin conductor. Observe however the following theorem (see \cite[Theorem 2.1]{Rib94} and the Remarque at pg 195 of \cite{Ser87})
\begin{theorem}\label{thm:ribet}
    Assume that $p\geq 3$. Suppose that a representation $\rho : G_\bb{Q}\rightarrow \gl(V)$ as in the beginning of the section arises from a form in $S_k(\Gamma_1(Mp^\alpha))$ with $(p,M)=1$. Then $\rho$ arises from a form in $S_k(\Gamma_1(N))$.
\end{theorem}
This means that we can always \say{strip away} powers of $p$ from the level, so the minimal level will necessarily be coprime to $p$. in view of Theorem \ref{thm:ribet}, Serre \cite{Ser87} defines the minimal level $N(\rho)$ as the prime to $p$ part of the global Artin conductor i.e.
\begin{equation*}
    N(\rho)=\prod_{l\neq p}l^{f(\rho,l)}
\end{equation*}
where (just as in characteristic 0) we have for all primes $l$ 
\begin{equation*}
    f(\rho,l)=\sum_{i=0}^\infty \frac{|G_i|}{|G_0|}\dim V/V_i
\end{equation*}
where the $G_i$ are the $i$-th ramification groups at $l$ of the field cut out by $\rho$.
\item (The minimal weight.) The definition of the minimal weight $k(\rho)$ is more involved since it depends on the behaviour of $\rho$ when restricted to an inertia group $I_p$ above $p$. For our purposes we only need the definition of $k(\rho)$ in one particular case. Put $$\omega_2=\epsilon_2 \bmod p \text{ and }\omega_2^\prime=\epsilon_2^\prime \bmod p$$ and observe that $\omega_2^\prime=\omega_2^p$. Suppose that $\rho\vert_{G_{\bb{Q}_p}}$ is irreducible and that $\rho\vert_{I_p}$ splits, then there exist some integers $0 \leq a < b \leq p-1$ such that 
\begin{equation*}
    \rho|_{I_p} \cong
    \begin{pmatrix}
        \omega_2^a(\omega_2')^b & 0 \\
        0 & (\omega_2')^a\omega_2^b
    \end{pmatrix}.
\end{equation*}
The minimal weight is defined as $$k(\rho) = 1 + pa + b.$$ 
\item (The nebentypus.) Observe that $\det\rho$ is a 1-dimensional representation of conductor dividing $pN(\rho)$ so putting $N=N(\rho)$, it can be viewed as a homomorphism from $\left(\bb{Z}/pN\bb{Z}\right)^\times$ to $\overline{\bb{F}}_p^\times$. The nebentypus $\chi(\rho)$ will be the restriction of $\det\rho$ to $\left(\bb{Z}/N\bb{Z}\right)^\times$. 
\end{itemize}
In the case $a_p=0$ we have the following result of Fontaine.
\begin{theorem}[Theorem 2.6, \cite{Edi92}]\label{thm:Fontaine}
    Let $f\in S_k(\Gamma_0(N),\chi)$ be a newform with $2\leq k\leq p+1$ and $a_p=0$ with $p\nmid N$. Then the representation $$ \rho : G_\bb{Q}\rightarrow \gl_2(\overline{\bb{F}}_p)$$ associated to $f$ is irreducible and $\rho\vert_{I_p}=\omega_2^{k-1}\oplus\omega_2^{p(k-1)}$.
\end{theorem}
The theorem of Fontaine allows us to evaluate the minimal weight in the case $a_p=0$. We see that under the hypotheses of Therorem \ref{thm:Fontaine}, $a=0$ and $b=k-1$ so the condition $0\leq a<b\leq p-1$ becomes $k\leq p$. If $k\leq p$ we have $k(\rho)=k$, if $k=p+1$ then
\begin{equation*}
    \rho|_{I_p} \cong
    \begin{pmatrix}
        (\omega^\prime_2)^p & 0 \\
        0 & \omega_2^p
    \end{pmatrix}
    =    \begin{pmatrix}
        \omega_2 & 0 \\
        0 & \omega_2^\prime
    \end{pmatrix}
\end{equation*}
and interchanging $\omega_2$ and $\omega_2^\prime$ we have $a=0$ and $b=1$ so $k(\rho)=2$. We conclude that if $2\leq k\leq p+1$ the minimal weights are either $k$ or $2$ depending on whether $k\leq p$ or $k=p+1$. In our examples we will always have $k\leq p$ so we will always be in the case $k(\rho)=k$.

We now discuss Mascot's method \cite{Mas22}. Fix a newform $f\in S_k(\Gamma_0(N),\chi)$, a prime $p$ that is \emph{unramified} and has \emph{residual degree 1} in $\bb{E}$. Consider the residual representation $$\rho(\mathfrak{p}):G_\bb{Q}\rightarrow \gl_2(\mathcal{O}_\bb{E}/\mathfrak{p})$$ at $\mathfrak{p}$ given by $f$ and let $l$ be a prime such that $l\mid N$ but $l^2\nmid N$. Assuming that $p\mid N(\rho)$, up to conjugacy we have
\begin{equation*}
    \rho(\mathfrak{p})_{\mid I_l}= 
    \begin{pmatrix}
        1 &\xi \\
        0 & {\chi\vert_{I_p}}\\
    \end{pmatrix}
\end{equation*}
for some nontrivial morphism $\xi:I_l\rightarrow \bb{F}_p$. If the $l$-part of $\chi$ is trivial we have that $\chi\vert_{I_p}=1$ and $\rho(\mathfrak{p})(I_l)$ is cyclic of order $l$. By \cite[Proposition 15]{Ser72} we have that the image of the residual representation contains $\operatorname{SL}_2(\bb{F}_p)=\sl_2(\mathcal{O}_\bb{F}/p\mathcal{O}_\bb{F})$.

In summary given a modular form $f\in S_k(\Gamma_0(N),\chi)$ and distinct primes $l,p$ such that $l\mid N$ but $l^2\nmid N$ and the $l$-part of $\chi$ is trivial, the representation at $p$ is irreducible and $p$ is unramified and with residual degree 1 in $\bb{E}$, then the image of the residual representation contains $\operatorname{SL}_2(\mathcal{O}_\bb{E}/p\mathcal{O}_\bb{E})$. By choosing a prime $r\nmid pN$ such that the characteristic polynomial of the Frobenius at $r$ is irreducible in $\bb{F}_p$, the residual representation is irreducible. In conclusion we have
\begin{proposition}[Theorem 5, \cite{Mas22}]\label{thm:mascot}
    Assume that $p\nmid N$ has ramification and residual degree 1 in $\mathbb{E}$. Suppose that there is a prime divisor $l$ of $N$ such that $l^2\nmid N$ and $l$ is coprime to the conductor of $\chi$, and another prime $r\nmid lN$ such that the polynomial $X^2-a_r(f)X+r^{k-1}\chi(r)\bmod {p}\in\mathbb{F}_p[X]$ is irreducible. Then either $\rho({p})$ also arises from a newform of level dividing $N/l$, or $\rho({p})(G_\mathbb{Q})$ contains $\sl_2(\mathbb{F}_p)$.
\end{proposition}
We list some forms satisfying the hypotheses of Theorem \ref{thm:mascot}.
\begin{itemize}
    \item[--] Consider the newform $f\in S_2(\Gamma_0(34),\chi)$ with LMFDB label 34.2.c.b and $q$-expansion
    \begin{equation*}
        f(q)=q + i q^2 - q^4 + (-i - 1) q^5 - i q^8 - 3i  q^9 + O(q^{10}).
    \end{equation*}
    It has Hecke field $\bb{E}=\bb{Q}(i)$ and taking $p=89$ we have $a_p=0$. Moreover $p$ is unramified and with residual degree 1 in $\bb{E}$. The nebentypus $\chi$ is the Dirichlet character modulo 34 of conductor 17 such that $\chi(3)=i$. Taking $l=3$ we see that $l\mid N$ but $l^2\nmid N$, and $l$ is coprime to the conductor of $\chi$. Taking $r=3$ we have $a_r=0$ and the characteristic polynomial of the Frobenius is
    \begin{equation*}
        X^2+3i.
    \end{equation*}
    The discriminant is $\Delta=-12i$ that is not a square in $\bb{F}_p$. We conclude that the representation is irreducible mod $p$ and that the image of $\rho(p)$ either contains $\sl_2(\mathcal{O}_\bb{E}/p\mathcal{O}_\bb{E})$ or there exists another newform of weight $2$ that produces the same residual representation. But there is no form of weight 2 of level 7 with the same nebentypus as $f$, so we conclude that the image contains $\sl_2(\mathcal{O}_\bb{E}/p\mathcal{O}_\bb{E})$.
    \item[--]
    Consider the newform $f\in S_3(\Gamma_0(26),\chi)$ with LMFDB label 26.3.d.a and $q$-expansion
    \begin{equation*}
        f(q)=q+(i + 1)q^2+2iq^4 + (-3i - 3)q^5 + (-2i + 2) q^7 +O(q^8)
    \end{equation*}
    It has Hecke field $\bb{E}=\bb{Q}(i)$ and taking $p=373$ we have $a_p=0$. Moreover $p$ is unramified and with residual degree 1 in $\bb{E}$. The nebentypus $\chi$ is the Dirichlet character modulo 26 of conductor 13 such that $\chi(15)=-i$. Taking $l=2$ we see that $l\mid N$ but $l^2\nmid N$, and $l$ is coprime to the conductor of $\chi$. Taking $r=11$ we have $a_r=6i+6$ and the characteristic polynomial of the Frobenius is
    \begin{equation*}
        X^2-(6i+6)X+121i.
    \end{equation*}
    The discriminant is $\Delta=-412i$ that is not a square in $\bb{F}_p$. We conclude that the representation is irreducible mod $p$ and that the image of $\rho(p)$ either contains $\sl_2(\mathcal{O}_\bb{E}/p\mathcal{O}_\bb{E})$ or there exists another newform of weight 3 that produces the same residual representation. The only possible form is $g\in S_3(\Gamma_0(13),\chi)$ with label 13.3.d.a but one easily checks that $f$ and $g$ are not congruent mod $p$. We conclude that the image contains $\sl_2(\mathcal{O}_\bb{E}/p\mathcal{O}_\bb{E})$.
    \item[--] 
    
    Consider the newform $f\in S_3(\Gamma_0(21),\chi)$ with LMFDB label 21.3.f.b and $q$-expansion
    \begin{equation*}
        f(q)=q -\zeta_6 q^2 + (\zeta_6 + 1) q^3 + (-3\zeta_6 + 3) q^4 + (3\zeta_6 - 6) q^5 + O(q^6)
    \end{equation*}
    It has Hecke field $\bb{E}=\bb{Q}(\sqrt{-3})$ and taking $p=103$ we have $a_p=0$. Moreover $p$ is unramified and with residual degree 1 in $\bb{E}$. The nebentypus $\chi$ is the Dirichlet character modulo 21 of conductor 7 such that $\chi(8)=1$ and $\chi(10)=\zeta_6$. Taking $l=3$ we see that $l\mid N$ but $l^2\nmid N$, and $l$ is coprime to the conductor of $\chi$. Taking $r=5$ we have $a_r=3\zeta_6-6$ and the characteristic polynomial of the Frobenius is
    \begin{equation*}
        X^2-(3\zeta_6-6)X+(-25\zeta_6 + 25).
    \end{equation*}
    The discriminant is $\Delta=73\zeta_6 - 73$ that is not a square in $\bb{F}_p$. We conclude that the representation is irreducible mod $p$ and that the image of $\rho(p)$ either contains $\sl_2(\mathcal{O}_\bb{E}/p\mathcal{O}_\bb{E})$ or there exists another newform of weight 3 that produces the same residual representation. But there is no newform of level 7 and weight 3 with the same nebentypus so we conclude that the image contains $\sl_2(\mathcal{O}_\bb{E}/p\mathcal{O}_\bb{E})$.
\end{itemize}
All of these forms satisfy the hypotheses of Theorem \ref{thm:main}, so their adelic representation has property (B).

\section*{Acknowledgments}
I want to thank Lea Terracini and Andrea Ferraguti who provided the insights and ideas that made this paper possible. I also want to thank Andrea Conti for the helpful discussions and the warm hospitality in Heidelberg where this project has been finalized and Francesco Amoroso who introduced me to the theory of heights.\\ The author is a member of the INDAM group GNSAGA.
\newpage
\printbibliography
\end{document}